\documentclass{amsart}
\pdfoutput=1
\usepackage[utf8]{inputenc} 

\usepackage[T1]{fontenc}    % use new font-encoding (to have bold
\usepackage{ae,aecompl}     % changes to Type 1 fonts
\usepackage{url}

\usepackage{amsmath}
\usepackage{amsfonts}
\usepackage{amssymb}
\usepackage{amsthm} % 
\usepackage{mathrsfs} % 
\usepackage{enumerate} 
\usepackage{verbatim}	%To comment parts
\usepackage[all]{xy}

\usepackage{tikz}
\usepackage{pgfplots}
\usetikzlibrary{calc}
\usetikzlibrary{shapes}

\usetikzlibrary{intersections}
\usetikzlibrary{patterns}

\usepackage{caption}
\usepackage[hyperfootnotes]{hyperref}

\theoremstyle{plain} % style plain
\newtheorem{thm}{Theorem}[section]
\newtheorem{cor}[thm]{Corollary}
\newtheorem{prop}[thm]{Proposition}
\newtheorem{lem}[thm]{Lemma}
\newtheorem{conject}{Conjecture}

\theoremstyle{definition}
\newtheorem{defi}[thm]{Definition}

\newtheorem{remark}[thm]{Remark}

\usepackage{color}
\definecolor{Ccolor}{rgb}{0,0.5,0}
\definecolor{Pcolor}{rgb}{1,0,0}
\definecolor{Ncolor}{rgb}{0,0,1}
\definecolor{lightgray}{rgb}{0.6,0.6,0.6}

\newcommand{\h}{\widehat}
\newcommand{\cone}{\operatorname{cone}}
\newcommand{\conv}{\operatorname{conv}}

\newcommand{\Span}{\operatorname{span}}
\newcommand{\Gar}{\operatorname{Gar}}
\newcommand{\dep}{\textnormal{dp}}

\newcommand{\supp}{\operatorname{supp}}
\newcommand{\Red}{\operatorname{Red}}
\newcommand{\Shi}{\textnormal{Shi}}
\newcommand{\B}{B}

\author[C.~Hohlweg]{Christophe~Hohlweg$^{\diamond}$}
\address[Christophe Hohlweg]{Universit\'e du Qu\'ebec \`a Montr\'eal\\
LaCIM et D\'epartement de Math\'ematiques\\ CP 8888 Succ. Centre-Ville\\
Montr\'eal, Qu\'ebec, H3C 3P8\\ Canada}
\email{hohlweg.christophe@uqam.ca}
\urladdr{http://hohlweg.math.uqam.ca}
\thanks{$^\diamond$supported by NSERC Discovery grant {\em Coxeter groups and related structures}.}

\author[P.~Nadeau]{Philippe Nadeau}
\address[Philippe Nadeau]{CNRS \& Institut Camille Jordan\\ Universit\'e Claude Bernard Lyon 1\\
69622 Villeurbanne Cedex\\ France}
\email{nadeau@math.univ-lyon1.fr}
\urladdr{http://math.univ-lyon1.fr/{\textasciitilde}nadeau}

\author[N.~Williams]{Nathan Williams}
\address[Nathan Williams]{Universit\'e du Qu\'ebec \`a Montr\'eal\\LaCIM\\ CP 8888 Succ. Centre-Ville\\
Montr\'eal, Qu\'ebec, H3C 3P8\\ Canada}
\email{nathan.f.williams@gmail.com}
\urladdr{http://thales.math.uqam.ca/~nwilliams/}

\keywords{Coxeter groups,  Garside shadows,  low elements, weak order,  dominance order, small roots, automaton and reduced words}

\subjclass[2010]{Primary 20F55; secondary 20F10; 05E15; 06F99}

\title[Automata, reduced words and Garside shadows in Coxeter groups]{Automata,  reduced words  and Garside shadows in Coxeter groups}

\begin{document}

\begin{abstract}
In this article, we introduce and investigate a class of finite deterministic automata that all recognize the language  of reduced words of a finitely generated Coxeter system  $(W,S)$. The definition of these automata is straightforward as it only requires the notion of {\em weak order} on $(W,S)$ and the related notion of {\em Garside shadows in $(W,S)$},   an analog of the notion of a Garside family. Then we discuss the relations between this class of automata and the canonical automaton built from Brink and Howlett's  small roots. We end this article by providing partial positive answers to two conjectures: (1) the automata associated to the smallest Garside shadow is minimal; (2) the canonical automaton is minimal if and only if the support of all small roots is spherical, i.e., the corresponding root system is finite. 
\end{abstract}

\date{\today}

%%%%%%%%%%%%%%%%%%%%%%%%%%%%%%%%%

 \maketitle

%%%%%%%%%%%%%%%%%%%%%%%%%%%%%%%
\section{Introduction}
%%%%%%%%%%%%%%%%%%%%%%%%%%%%%%%

In this article, we introduce and investigate a class of finite deterministic  automata that recognize the language $\Red(W,S)$ of reduced words of a finitely generated Coxeter system  $(W,S)$.  The definition of   these automata is straightforward, requiring only the notion of {\em (right) weak order $\leq_R$} on $(W,S)$~\cite{Bj84,BjBr05}  and the related notion of {\em Garside shadows},  introduced by M.~Dyer and the first author in~\cite{DyHo15} as an analog of the notion of a Garside family in a monoid; see~\cite{DDH14,De15} and the references therein.  For general definitions and properties, we refer the reader to~\cite{Sa09} regarding  automata and to~\cite{BjBr05, Hu90}  regarding Coxeter groups.

\smallskip

 A {\em Garside shadow in $(W,S)$} is a subset $B\subseteq W$  that contains $S$ and is closed under join (for the right weak order) and  by  taking suffixes.   In~\cite{DyHo15}, the authors show that finite Garside shadows exist in any Coxeter  system  $(W,S)$.  Let $\B$ be a finite Garside shadow in $(W,S)$. So  $\bigvee X\in B$ for any {\em bounded} subset $X$ of $\B$, i.e., a subset that has an upper bound.   Therefore, the following projection from $W$ to $B$  is well-defined:
 \begin{eqnarray*}
 \pi_\B : W&\to& \B\\
 w&\mapsto&\bigvee\{g\in \B\,|\, g\leq_R w\}
\end{eqnarray*}

We denote by $\ell:W\to \mathbb N$ the {\em length function} of the Coxeter system $(W,S)$. 

\begin{defi}
  \label{def:automata}
We define a finite deterministic automaton $\mathcal A_\B(W,S)$ over the alphabet $S$ 
 as follows:
\begin{itemize}
  \item the set of states is $\B$;
  \item the initial state is the identity $e$ of $W$, and all states are final;
  \item the transitions are: $x \overset{s}{\rightarrow} \pi_\B(sx)$ whenever $\ell(sx)>\ell(x)$.
\end{itemize}
\end{defi}

 Since the intersection of Garside shadows is again a Garside shadow, there is a smallest Garside shadow $\tilde S$ in $(W,S)$.  As a first example, the finite automaton built out of the smallest Garside shadow~$\tilde S$ for the infinite dihedral group is shown in Figure~\ref{fig:InfiniteDi}.  Further examples are given in \S\ref{sse:Shi} and in Figures~\ref{fig:a1_automaton} and~\ref{fig:a2_c2_automata}.
%\COMMC{Nathan, could you put a reference to the figures you will be including here?}\COMMN{Done.} 

Our main result is that $\mathcal A_\B(W,S)$ recognizes the language of reduced words  of $(W,S)$.

\begin{thm}\label{thm:Main} If $\B$ is a finite Garside shadow in $(W,S)$, then the finite deterministic automaton $\mathcal A_\B(W,S)$ recognizes the language $\Red(W,S)$. 
\end{thm}

 Theorem~\ref{thm:Main}  is proved in \S\ref{se:Auto}. In \S\ref{se:Can}, we show that an inclusion $\B \subseteq C$ of Garside shadows induces a surjective morphism $\mathcal A_C(W,S) \to \mathcal A_\B(W,S)$ between their associated automata.  The smallest Garside shadow being finite~\cite[Corollary 1.2]{DyHo15}, we are led to the following conjecture.

\begin{conject}\label{conj:1}  The automaton $\mathcal A_{\tilde S}(W,S)$ is the minimal automaton recognizing $\Red(W,S)$. 
\end{conject}

\smallskip

 Using Sage~\cite{Sage-Combinat,sage}, we checked that  Conjecture~\ref{conj:1} holds for all Coxeter groups~$W$ of rank at most~$4$ whose corresponding Coxeter graph $\Gamma_W$ has edge labels less than~$10$; see Remark~\ref{rem:Sage} for more details.

Our initial motivation for this work was to provide a  purely  combinatorial definition for an automaton that recognizes the language of reduced words. Indeed, as we now recall,  all previously-defined automata recognizing $\Red(W,S)$ require the introduction of an auxiliary geometric representation and root system.   

%\begin{commN}
%In H.~Eriksson's thesis (page 52), he claims that what appears to be the canonical automaton for $\widetilde{A}_2$ is due to A.~Bj\"orner in 1990, for $\widetilde{A}_3$ to himself, and by R.~Stanley for $\widetilde{A}_n$.  P.~Headley established the result for all other affine groups, using the Shi arrangement (he also shows that $\widetilde{B}_2\equiv\widetilde{C}_2$ is not minimal on page 33 of his thesis).  H.~Eriksson argues (correctly, it seems to me) that the result that the language of reduced expressions is regular for Coxeter groups should be attributed to all of B.~Brink, R.~Howlett, M.~Davis, and M.~Shapiro.  I would be in favor of emphasizing this.  Christophe, I believe you are right to cite H.~Eriksson's thesis for the canonical automata (Bj\"orner-Brenti cite him for the sections of their book)---good scholarship!
%\end{commN}
%\begin{commC} Thanks for the 'good scholarship' ;) It seems that this compliment is in order for you too: good work to have dig out Headley's thesis. And I do agree, enjoy and like it a lot to try to have the straighter record possible for the history of the subject... looks like we do a bit of Indiana Jones' work :D
%\end{commC}

In 1993, B.~Brink and R.~Howlett~\cite{BrHo93} showed that finitely-generated Coxeter groups are automatic, in the sense of~\cite{Ep+92}, thereby filling a gap in the proof of the ``Parallel Wall Theorem'' of~\cite{DaSh91}.  For each Coxeter system $(W,S)$, they provided a {\em word-acceptor}---that is, a finite automaton that recognizes the language of lexicographically minimal reduced words in $W$.   This particular automaton is built using their notion of {\em small roots}, and therefore requires a geometric representation of $(W,S)$ and its associated root system. In a series of articles~\cite{Ca94,Ca95,Ca02,Ca08}, Casselman explains how to perform practical computations in Coxeter groups using Brink and Howlett's word-acceptor. 

  We are often interested in {\em all} reduced words, not only those that are lexico\-gra\-phically-ordered; see for instance~\cite{Sh15}. In his thesis~\cite{Er94}, H.~Eriksson studied a finite deterministic automaton $\mathcal A_{0}(W,S)$ over $S$  that recognizes the language $\Red(W,S)$.  The automaton $\mathcal A_{0}(W,S)$ is called the {\em canonical automaton} in~\cite[\S4.8]{BjBr05}, and is built using B.~Brink and R.~Howlett's technology of small roots. An immediate consequence is that the language $\Red(W,S)$  is  regular, a result we recover in Theorem~\ref{thm:Main}.  In particular, the generating function for the number of reduced words in $(W,S)$ with respect to their length is a rational function. 
   
For $n\in \mathbb N$, the canonical automaton was extended, replacing small roots with $n$-small roots, in~\cite{Ed09} and~\cite{DyHo15} to the $n$-canonical automaton $\mathcal A_n(W,S)$. We recall these notions in \S\ref{se:Can}, and discuss morphisms between $\mathcal A_n(W,S)$
 and the automata $\mathcal A_\B(W,S)$ arising from certain finite Garside shadows $\B$. In particular, we show in Corollary~\ref{cor:MinSurj} that any $n$-canonical automaton surjects into the automaton $\mathcal A_{\tilde S}(W,S)$,  providing evidence for  Conjecture~\ref{conj:1}.

Both H.~Eriksson~\cite[Theorem 80]{Er94} and P.~Headley~\cite[Theorem V.8]{He94} prove that in type $\widetilde{A}_n$, the canonical automaton $\mathcal A_0(\widetilde{A}_n,S)$ is minimal.  Furthermore, they note that $\mathcal A_0(W,S)$ is {\em not} minimal for general affine groups $W$. 
 
 We conjecture  a necessary  condition for  the canonical automaton to be minimal.  The sufficient condition is shown in  Proposition~\ref{prop:spherical_implies_minimaL}.

\begin{conject}\label{conj:2}
Let $W$ be irreducible.  Then $\mathcal{A}_{0}(W,S)$ is minimal if and only if $\Sigma=\Phi^+_{sph}$, where $\Phi^+_{sph}$ denotes the set of roots whose support is a finite standard parabolic subgroup.
\end{conject}

Since $\mathcal A_0(W,S)$ surjects onto $\mathcal A_{\tilde S}(W,S)$ (Corollary~\ref{cor:MinSurj}), Conjecture~\ref{conj:2} implies Conjecture~\ref{conj:1} for Coxeter systems for which $\Sigma=\Phi^+_{sph}$.  In  \S\ref{se:Min},  we  prove  Conjecture~\ref{conj:2}  in the following cases.

\begin{thm}\label{thm:CanMin} 
Conjecture~\ref{conj:2} holds in each of the following cases: 
\begin{enumerate}
\item $W$ is finite.
\item $W$ is right-angled, i.e. $m_{st}=2$ or $\infty$ for all $s\neq t$
\item $\Gamma_W$ is a complete graph, i.e. $m_{st}>2$ for all $s\neq t$.
\item $W$ is of type $\widetilde{A_{n-1}}$.
\item $W$ has rank $3$.
\end{enumerate}
In the first four cases, $\Sigma=\Phi^+_{sph}$ and $\mathcal{A}_{0}(W,S)$ is minimal.
\end{thm}

 We also checked that Conjecture~\ref{conj:2} holds if $W$ has rank $4$ and $m_{st}<10$ for all $s \neq t$; see Remark~\ref{rem:Sage} for more details.

%\NEWN In \S\ref{se:Evidence}, we show that Conjecture~\ref{conj:2} holds for all Coxeter groups of rank at most four and edge labels at most nine.\WEN

When $(W,S)$ is an affine Coxeter system, P.~Headley described a remarkable connection between the canonical automaton and the {\em Shi arrangement}~\cite{Sh86}: the states of $\mathcal A_{0}(W,S)$ are in bijection with the (minimal elements in the) connected regions of the complement of the Shi arrangement for $(W,S)$~\cite{He94}.  The same relationship holds for the states of $\mathcal A_{n}(W,S)$ and the regions of the $n$-Shi arrangement, as we outline in \S\ref{sse:Shi}.

%%%%%%%%%%%%%%%%%%%%%%%%%%%%%%%
\section{Garside shadow automata}\label{se:Auto}
%%%%%%%%%%%%%%%%%%%%%%%%%%%%%%%

Fix $(W,S)$ a Coxeter system with length function $\ell:W\to\mathbb N$. The {\em rank} of $W$ is the cardinality of the set of {\em simple reflections}  $S$.  A word $s_1\cdots s_k$ on the alphabet $S$ is a {\em reduced word for $w\in W$} if  $w=s_1\cdots s_k$ and $k=\ell(w)$. For $u,v,w\in W$, we say that: 

\begin{itemize}
\item {\em $w=uv$ is reduced} if $\ell(w)=\ell(u)+\ell(v)$, i.e., the concatenation of any reduced word for~$u$ with any reduced word for~$v$ is a reduced word for $w$; 
\item {\em $u$ is a prefix of $w$} if a reduced word for $u$ is a prefix of a reduced word for~$w$;
\item {\em $v$ is a suffix of $w$} if a reduced word for $v$ is a suffix of a reduced word for~$w$.
\end{itemize}
Observe that if $w=uv$ is reduced, then  $u$ is a prefix of $w$ and $v$ is a suffix of $w$. The subset  $D_L(w)=\{s\in S\,|\, \ell(sw)<\ell(w)\}$ of $S$ is called the {\em left descent set of $w\in W$}. The descent set plays an important role in the study of reduced words since it coincides with the set of the possible first letters  of reduced words of an element $w\in W$; see~\cite{BjBr05}. 

The {\em standard parabolic subgroup $W_I$} is the subgroup of $W$ generated by $I\subseteq S$. It is well-known that $(W_I,I)$ is itself a Coxeter system and that the length function $\ell_I:W_I\to\mathbb N$ is the restriction of $\ell$ to $W_I$.   Moreover, $W_I$ is finite if and only if it  contains   a {\em longest element}, which is then unique and is denoted by~$w_{\circ,I}$. 

 The set $X_I:=\{x\in W\,|\,\ell(sx)>\ell(x),\ \forall s\in I \}$ is the set of {\em minimal-length coset representatives for the coset $W_I\backslash W$}.  For any $w\in W$, there is a unique decomposition $w=w_Iw^I$, with $w_I w^I$ reduced; see~\cite[Proposition 2.4.4]{BjBr05}. See \cite{Hu90,BjBr05} for more details.

%%%%%%%%%%%%%%%%%%%%%%%%
\subsection{Weak order and Garside shadows}  
The {\em (right) weak order}  is the order on~$W$ defined by $u\leq_Rv$ if $u$ is a prefix of $v$. Since we only consider the right weak order in this article, we only use from now on the term {\em weak order}. The weak order gives a natural orientation of the Cayley graph of  $(W,S)$:  for $w\in W$ and $s\in S$, we orient an edge $w\to ws$ if $w\leq_Rws$.  We recall the following  well-known useful properties linking descent sets and weak order, which is a rephrasing of part of~\cite[Proposition 3.1.2]{BjBr05}.

\begin{lem}\label{lem:Desc} Let $u,v\in W$ and $s\in S$.
\begin{enumerate}[(a)]
\item $s\in D_L(u)$ if and only if $s\leq_R u$.
\item If $s\in D_L(u)\cap D_L(v)$, then $u\leq_R v$ if and only if $su\leq_R sv$.
\item  If  $s\notin D_L(u)$ and $s\notin D_L(v)$, then $u\leq_R v$ if and only if $su\leq_R sv$.
\end{enumerate}
\end{lem}

A.~Bj\"orner~\cite[Theorem~8]{Bj84} proved that the weak order $(W,\leq_R)$ is a complete meet semilattice: for any $A\subseteq W$, there exists an infimum $\bigwedge A\in W$, also called the {\em meet} of $A$; see~\cite[Chapter~3]{BjBr05}.

  A subset $X\subseteq W$ is  {\em bounded in $W$} if there exists a  $g\in W$ such that $x\leq_Rg$ for any $x\in X$. Therefore, any bounded subset $X\subseteq W$ admits a least upper bound $\bigvee X$ called {\em the join of $X$}: 
$$
\bigvee X=\bigwedge\{g\in W\,|\, x\leq_R  g,\, \forall x\in X\}.
$$ 
When $W$ is finite, any element $w\in W$ is a prefix of the longest element~$w_\circ$,  so that $W$ itself is bounded. In fact,  $(W,\leq_R)$ turns out to be a complete ortholattice; see~\cite[Corollary 3.2.2]{BjBr05}.

\begin{defi}[{\cite{DyHo15}}]
  \label{def:garside_shadow}
 A subset $\B\subseteq W$ is a {\em Garside shadow in $(W,S)$} if $\B$ contains $S$ and:
\begin{enumerate}[(i)]
\item $\B$ is closed under join in the weak order: if $X\subseteq \B$ is bounded, then $\bigvee X\in \B$;
\item $\B$ is closed under  taking suffixes: if $w\in \B$, then any suffix of $w$ is also in $\B$.
\end{enumerate}
\end{defi}

Since a standard parabolic subgroup $W_I$ with its canonical set of generators $I\subseteq S$ forms a Coxeter system, it is natural to say that a subset $\B \subseteq W_I$ is a Garside shadow of $(W_I,I)$ if $\B$ contains $I$ and verifies Conditions~(i)--(ii)  of Definition~\ref{def:garside_shadow}.  Note that if $\B$ is a Garside shadow in $(W,S)$, then $\B \cap W_I$ is a Garside shadow in $(W_I,I)$~\cite[Remark~2.5(c)]{DyHo15}. Since the intersection of  Garside shadows is again a Garside shadow, there exists a smallest Garside shadow of $(W,S)$ containing $X\subseteq W$, which we denote by $\Gar_S(X)$.  In~\cite[Corollary 1.2]{DyHo15}, Dyer and the first author show that the smallest Garside shadow
\[
\tilde S:=\Gar_S(S)
\]
is finite.  The automaton constructed from the smallest Garside shadow $\tilde S$ of the infinite dihedral group is illustrated in Figure~\ref{fig:InfiniteDi}. 

\begin{remark}\label{rem:Smallest} The finiteness of $\tilde S$ is shown  in~\cite{DyHo15} using the geometry of the root system. A direct computational proof is still open. The problem of computing $\tilde S$ relies on finding an efficient criterion for a subset of $W$ to be bounded. 
\end{remark}
\begin{figure}[h!]
\hspace*{\fill}
\begin{tikzpicture}
	[scale=1,
	 node distance=4mm,
	 pointille/.style={dashed},
	 axe/.style={color=black, very thick},
	 sommet/.style={inner sep=2pt,circle,draw=blue!75!black,fill=blue!40,thick,anchor=west},
	 sommetg/.style={inner sep=2pt,circle,draw=red!75!black,fill=red!40,thick,anchor=west}]
	 
\node[sommetg]  (id)    [label=below:{\small{$e$}}]          at (0,0)    {};
\node[sommetg]  (1)    [label=left:{\small{$s$}}]         at (-1,1)   {} edge[thick,<-] (id);
\node[sommetg]  (2)    [label=right:{\small{$t$}}]        at (1,1)    {} edge[thick,<-] (id);
\node[sommet]  (12)   [label=left:{\small{$st$}}]      at (-1,2)   {} edge[thick,<-] (1);
\node[sommet]  (21)   [label=right:{\small{$ts$}}]     at (1,2)    {} edge[thick,<-] (2);
\node[sommet]  (121)  [label=left:{\small{$sts$}}]    at (-1,3)    {} edge[thick,<-] (12);
\node[sommet]  (212)  [label=right:{\small{$tst$}}]    at (1,3)    {} edge[thick,<-] (21);

\draw[pointille] (121) -- +(0,1);
\draw[pointille] (212) -- +(0,1);
\end{tikzpicture}
\hspace*{\fill}
\begin{tikzpicture}
	[scale=1,
	 node distance=4mm,
	 pointille/.style={dashed},
	 axe/.style={color=black, very thick},
	 sommet/.style={inner sep=2pt,circle,draw=blue!75!black,fill=blue!40,thick,anchor=west},
	 sommetg/.style={inner sep=2pt,circle,draw=red!75!black,fill=red!40,thick,anchor=west}]
	 
\node[sommetg]  (id)   at (0,1)    {$e$};
\node[sommetg]  (1)    at (-1,3)   {$s$} edge[thick,bend right,<-] node[left] {$s$} (id); 
\node[sommetg]  (2)    at (1,3)    {$t$} edge[thick,bend left,<-] node[left] {$t$} (id);
\draw (1) edge[thick,bend right,<-]  node[below] {$s$} (2);
\draw (2) edge[thick,bend right,<-]  node[above]  {$t$} (1);

\end{tikzpicture}
\hspace*{\fill}
\caption{The weak order on  the infinite dihedral group $\mathcal D_\infty$ and the automaton associated to the smallest Garside shadow $\{e,s,t\}$, which is represented by red vertices.}
\label{fig:InfiniteDi}
\end{figure}
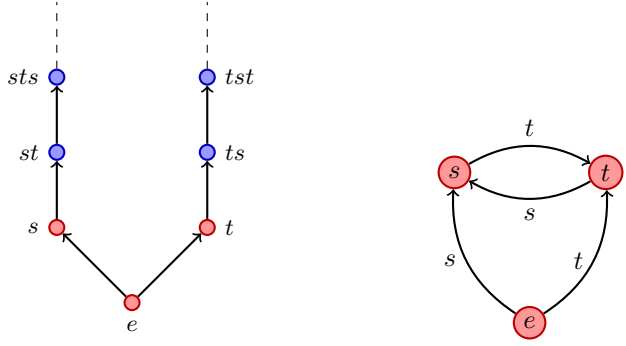
%

%%%%%%%%%%%%%%%%%%%%
\subsection{Garside Shadow Projections}

\begin{defi}
\label{def:garside_projection}
Let $\B$ be a Garside shadow in $(W,S)$.  We call the surjection
\begin{eqnarray*}
 \pi_\B : W&\to& \B\\
 w&\mapsto&\bigvee\{g\in \B\,|\, g\leq_R w\},
\end{eqnarray*}
the {\em  $\B$-projection}.
\end{defi}

Since $S\subseteq \B$, the set $\{g\in \B\,|\, g\leq_R w\}$ is non-empty and bounded for any~$w\in W$. Together with Condition~(i) of Definition~\ref{def:garside_shadow}, this implies that the $\B$-projection is well-defined.  Note that $\pi_B(w)$ can be characterized as the unique longest prefix of $w$ which belongs to $B$.

\begin{prop}\label{prop:Proj1} Let $\B$ be a Garside shadow in $(W,S)$ and $u,w\in W$, then:
\begin{enumerate}[(a)]
\item $\pi_\B\circ \pi_\B= \pi_\B$;
\item $\pi_\B(w)\leq_R w$, with equality holding if and only if $w\in \B$; 
\item If $u\leq_R w$, then $\pi_\B(u)\leq_R \pi_\B(w)$.
\end{enumerate}
\end{prop}

\begin{proof}
Properties (a) and (b) are clear from the definition.  For (c), if $u\leq_R w$ then  $x\leq_R u$ implies $x \leq_R w$ for any $x \in \B$, from which we conclude the proposition.
\end{proof}

The next proposition states that left descent sets are invariant under Garside shadow projections. 

\begin{prop}\label{prop:Proj2} Let $\B$ be a Garside shadow in $(W,S)$ and $w\in W$. Then  $D_L(w)=D_L(\pi_\B(w))$, and $s\pi_\B(w)\leq_R sw$ for any $s\in S$.
\end{prop}
\begin{proof}  We first show $D_L(w)=D_L(\pi_\B(w))$.   Let $r\in D_L(\pi_\B(w))$. By Lemma~\ref{lem:Desc}(a) and  Proposition~\ref{prop:Proj1}(b), we have $r\leq_R \pi_\B(w)\leq_R w$. So $D_L(\pi_\B(w))\subseteq D_L(w)$.

 Conversely, let $s\leq_R w$. Since $S\subseteq B$,  we have $s=\pi_B(s)$ by Proposition~\ref{prop:Proj1}(b).  Then   $s\leq_R \pi_\B(w)$  by Proposition~\ref{prop:Proj1}(c). So $D_L(w)\subseteq D_L(\pi_\B(w))$. 

We now show $s\pi_\B(w)\leq_R sw$ for any $s\in S$.   We write  $D:=D_L(w)=D_L(\pi_\B(w))$.   If $s\notin D$, then   the result follows from Proposition~\ref{prop:Proj1}(b) and Lemma~\ref{lem:Desc}(c). If $s\in D$, the statement follows from Proposition~\ref{prop:Proj1}(b) and Lemma~\ref{lem:Desc}(b).

\end{proof}

\begin{remark}\label{rem:proj_uses} Definition~\ref{def:garside_projection} of $\pi_\B$ and the proof of Proposition~\ref{prop:Proj2} require only the conditions  $S\subseteq \B$ and Condition~(i) from  Definition~\ref{def:garside_shadow}. The proof of Proposition~\ref{prop:Proj1} only uses properties of the weak order.
\end{remark}

The next result is crucial for proving Theorem~\ref{thm:Main}.

\begin{prop}\label{prop:Proj3} Let $\B$ be a Garside shadow in $(W,S)$. Let  $w\in W$ and $s\in S$ such that $s\notin D_L(w)$. Then $\pi_\B(sw)=\pi_\B(s\pi_\B (w))$.
\end{prop}
\begin{proof} We have $s\pi_\B(w)\leq sw$ by Proposition~\ref{prop:Proj2}. Therefore, by Proposition~\ref{prop:Proj1}(c), $\pi_\B(s\pi_\B (w))\leq_R \pi_\B(sw)$.  To complete the proof, we will show that $\pi_\B(sw)\leq_R \pi_\B(s\pi_\B (w))$, which is equivalent by Proposition~\ref{prop:Proj1} (a), (b), and (c) to the statement that $\pi_\B(sw)\leq_R s\pi_\B (w)$. We prove this last relation as follows.  Using Proposition~\ref{prop:Proj2}, we see that $s\in D_L(sw)=D_L(\pi_B(sw))$, so that since $\pi_B(sw) \leq_R sw$ by Proposition~\ref{prop:Proj1}(b), Lemma~\ref{lem:Desc}(b) allows us to conclude that $s\pi_B(sw)\leq_R w$.  Now $s\pi_B(sw)\in B$ because $B$ is closed under taking suffixes by Definition~\ref{def:garside_shadow}(ii), so that $\pi_B(s\pi_B(sw))=s\pi_B(sw)\leq_R\pi_B(w)$ by Proposition~\ref{prop:Proj1}(b,c).  Multiplying both sides by $s$ and using Lemma~\ref{lem:Desc}(c) gives $\pi_B(sw)\leq_R s\pi_B(w)$.
\end{proof}

\begin{remark} The proof of Proposition~\ref{prop:Proj3} requires all the conditions from  Definition~\ref{def:garside_shadow}.
\end{remark}

\begin{cor}\label{cor:Proj} Let $v\in W$ and $s_1,\dots,s_k\in S$ such that $s_k\cdots s_1 v$ is reduced. Then   
$$
\pi_\B(s_k\cdots s_1 v)=\pi_\B(s_k\pi_\B(s_{k-1}\pi_\B(\cdots s_2\pi_\B(s_1\pi_B(v))))) =\pi_B(s_k\cdots s_1 \pi_B(v)).
$$
In particular:
\begin{enumerate}[(a)]
\item If $w=s_k\cdots s_1$ is a reduced word, then
$$
\pi_\B(w)=\pi_\B(s_k\cdots s_1)=\pi_\B(s_k\pi_\B(s_{k-1}\pi_\B(\cdots s_2\pi_\B(s_1))));
$$
\item  If $uv$ is reduced, $u\in W$, then $\pi_\B(u\pi_\B(v))=\pi_\B(uv)$.
\end{enumerate}
\end{cor}
\begin{proof} We prove the first equality by induction on $k>0$. The case $k=1$ is Proposition~\ref{prop:Proj3}. Now assume the property for $k-1>0$. Then since $s_k s_{k-1} \cdots s_1v$ is reduced, we have $s_k\notin D_L(s_{k-1} \cdots s_1v)$. By Proposition~\ref{prop:Proj3},  we obtain
$$
\pi_\B(s_k\cdots s_1v)=\pi_\B(s_k\pi_\B(s_{k-1}\cdots s_1v)).
$$
 By induction, $\pi_\B(s_{k-1}\cdots s_1v)=\pi_\B(s_{k-1}\pi_\B(\cdots s_2\pi_\B(s_1\pi_B(v))))$.

Now for the second equality, observe that $\pi_B(v)$ is a prefix of $v$ by Proposition~\ref{prop:Proj1}(b). Therefore 
 $s_k\cdots s_1 \pi_B(v)$ must be reduced,  since $s_k\cdots s_1 v$ is reduced. We conclude by applying the first equality to $\pi_B(v)$, recalling that $\pi_B$ is a projection. In particular (a) is obtained by taking $v=e$ and  (b) by considering a reduced word $s_k\cdots s_{1}$ for $u$.
\end{proof}

%%%%%%%%%%%%%%%%%%%%
\subsection{Garside Shadow Automata and Proof of Theorem~\ref{thm:Main}}

Before proving Theorem~\ref{thm:Main}, we recall some terminology about automata theory; see~\cite{Sa09}. A {\em finite deterministic automaton} $\mathcal{A}$ over
the alphabet $S$ is a quadruple $(Q,q_0,F,\delta)$ where $Q$ is a
finite set of \emph{states}, $q_0\in Q$ is the {\em initial state},
$F\subseteq Q$ is the set of \emph{final states}, and $\delta$ is a
partial function $Q\times S\to Q$. If $\delta(q,s)=q'$ then
$q\overset{s}{\rightarrow}q'$ is a \emph{transition}.  An automaton  $\mathcal{A}$
can thus be seen as a directed graph on the vertex set $Q$ with edges
labeled by elements of $S$ such that for any $q,s$ there is at most
one edge with source $q$ and label $s$.

 For an automaton $\mathcal A$, one naturally extends $\delta$ to a partial
function $Q\times S^*\to Q$. A word $s_1\cdots s_k\in S^*$ is
\emph{accepted} by $\mathcal{A}$ if $\delta(q_0,s_1\cdots s_k)$ is
defined and is in $F$. The set of all words accepted by $\mathcal{A}$
is the {\em language recognized} by $\mathcal{A}$ and denoted by
$\mathcal{L}(\mathcal{A})$. Languages $L\subseteq S^*$ occurring in
this way are called  \emph{regular}, and it is a fundamental theorem
of Kleene that the class of such languages coincides with the class of \emph{rational languages}; see \cite[Theorem~2.1]{Sa09}.

 Let $\B$ be a Garside shadow.   Recall from Definition~\ref{def:automata} in  the introduction that the automaton $\mathcal A_\B(W,S)$  is defined by:
 \begin{itemize}
  \item the set of states is $\B$;
  \item the initial state is the identity $e$ of $W$, and all states are final;
  \item the transitions are: $x \overset{s}{\rightarrow} \pi_\B(sx)$ whenever $s\notin D_L(x)$.
\end{itemize}
  
 We denote $\mathcal A_\B:= \mathcal A_\B(W,S)$ if there is no possible confusion. We prove now that $\mathcal A_\B(W,S)$ recognizes the language $\Red(W,S)$ of reduced words in $(W,S)$.

\begin{proof}[Proof of Theorem~\ref{thm:Main}]  We prove the theorem by induction. Let $\mathcal{P}(k)$ ($k\in\mathbb N$) be the following property:  
	
	\emph{For any sequence $s_1,\ldots,s_k$ of  simple reflections, $s_k\cdots s_1$ is reduced if and only if there is a path in $\mathcal A_\B$ starting at the initial state $e$ with edges labeled successively by $s_1,\ldots,s_k$. The final state of such a path is $\pi_\B(s_k\cdots s_1)$.}
	
By definition of $\mathcal A_\B$, properties $\mathcal P(0)$ and $\mathcal P(1)$ are easily seen to be true. Now let $k>1$ be such that $\mathcal{P}(i)$ holds for all $i<k$, and consider any  sequence  $s_1,\ldots,s_k\in S$.  Let $w_{j} := s_j \cdots s_1$, and let $x_j := \pi_\B(s_j x_{j-1})$.

We first show that the sequence of edge labels for a path in $\mathcal A_\B$ is reduced.  Suppose there is a path in $\mathcal A_\B$ \[e \overset{s_1}{\rightarrow} x_1 \overset{s_2}{\rightarrow} x_2 \overset{s_3}{\rightarrow} \cdots \overset{s_{k-1}}{\rightarrow} x_{k-1} \overset{s_k}{\rightarrow} x_k\] from the state $e$ to the state $x_k$, with edges labeled by $s_1,\ldots,s_k$.  By induction, $s_{k-1}\cdots s_1$ is reduced, so that by Corollary~\ref{cor:Proj} $\pi_\B(w_{k-1})=x_{k-1}$.  Since $x_{k-1} \overset{s_k}{\rightarrow} x_k$ is an edge in the automaton $\mathcal A_B$, $s_k\notin D_L(x_{k-1})$ by definition.  Therefore $s_k\notin D_L(w_{k-1})$, since $D_L(x_{k-1})=D_L(w_{k-1})$ by Proposition~\ref{prop:Proj2}.  In particular, since $s_{k-1}\cdots s_1$ is reduced, $s_kw_{k-1} = s_k s_{k-1} \cdots s_1$ is also reduced.

We now show that any reduced word $s_k s_{k-1} \cdots s_1$ gives a path in $\mathcal A_\B$ from $e$, with the desired edge labels and ending state.  For the sake of contradiction, suppose that the sequence $s_1,\ldots,s_k$ does not define a path in $\mathcal A_\B$.

This gives rise to two cases, both of which lead to a contradiction of our initial assumption that $s_k s_{k-1} \cdots s_1$  is reduced.  If the initial sequence $s_1,\ldots,s_{k-1}$ does not define a path in $\mathcal A_\B$,  then, by induction $s_{k-1}\cdots s_1$ is not reduced, contradicting our assumption.  Otherwise, the sequence $s_1,\ldots,s_{k-1}$ ends at the state $x_{k-1}$ and, by induction, $s_{k-1}\cdots s_1$ is reduced.  In particular, $\pi_\B(w_{k-1})=x_{k-1}$ by Corollary~\ref{cor:Proj}.  Since $s_1,\ldots,s_k$ does not define a path, $s_k\in D_L(x_{k-1})$, so that $s_k\in D_L(w_{k-1})$ by Proposition~\ref{prop:Proj2}.  But then $s_{k} s_{k-1}\cdots s_1 = s_k w_{k-1}$ is not reduced, which again contradicts our initial assumption.
\end{proof}

\begin{remark} Neither the definition of $\mathcal A_\B(W,S)$, the definition of $\pi_\B$  nor the proof of Theorem~\ref{thm:Main} requires $\B$ to be finite. However we chose to state the result for finite Garside shadows in Theorem~\ref{thm:Main}, since those produce finite automata. 
\end{remark}

%%%%%%%%%%%%%%%
\subsection{Root systems and inversion sets}\label{sse:Root}  Before studying the relation between Garside shadow automata and standard parabolic subgroups, we need to introduce a geometric representation and a root system for $(W,S)$.

Recall that a quadratic space $(V,B)$ is a data of a real vector space $V$ with a symmetric bilinear form $B$. The group $O_B(V)$ is the group consisting of all linear maps that preserve $B$. For any non-isotropic vector $\alpha\in V$, i.e., $B(\alpha,\alpha)\not = 0$, we associate a $B$-reflection $s_\alpha$ given by the formula $s_\alpha(v)=v-2\frac{B(\alpha,v)}{B(\alpha,\alpha)}\alpha$, for all $v\in V$. 

We consider now a {\em geometric representation of $(W,S)$}, i.e., a faithful representation of $W$ as a subgroup of  $O_B(V)$, where $S$ is mapped into the set of $B$-reflections associated to a {\em simple system} $\Delta=\{\alpha_s\,|\, s\in S\}$ ($s=s_{\alpha_s}$). Then the $W$-orbit $\Phi=W(\Delta)$ is a {\em root system} with {\em positive roots} $\Phi^+=\cone_\Phi(\Delta)$ and {\em negative roots $\Phi^-=-\Phi^+$}, where $\cone(X)$ is the set of nonnegative linear combination of vectors in $X\subseteq V$ and $\cone_\Phi(X)=\cone(X)\cap \Phi$; see  \cite[\S1]{HoLaRi14} for more details.

We recall now some useful well-known results linking roots and reduced words in $(W,S)$. 

The {\em left inversion set of $w\in W$} of $w\in W$ is defined by $N(w):=\Phi^+\cap w(\Phi^-)$.  The following proposition  may be found in~ \cite[\S2.3-\S2.5]{HoLa15}; part (b)  is due to M~Dyer~\cite{Dy11}.

\begin{prop}\label{prop:Inv} The map $N:(W,\leq_R)\to (\mathcal P(\Phi^+),\subseteq )$ is a poset monomorphism.  Furthermore: 
\begin{enumerate}[(a)]
\item  For any $u,w\in W$,  $u\leq_R w$ if and only if $N(u)\subseteq N(w)$;
\item For any bounded $X\subseteq W$, $N(\bigvee X)=\cone_\Phi\left(\bigcup_{x\in X} N(x)\right)$.
\end{enumerate}
\end{prop}

 If  $I\subseteq S$, then $\Delta_I:=\{\alpha_s\,|\, s\in I\}$ is a simple system with root system $\Phi_I:=W_I(\Delta_I)$ and positive root system  $\Phi^+_I:=\Phi_I\cap\Phi^+$ for the standard parabolic subgroup $W_I$. The following statement is well-known; we include a proof here for completeness. 
%\begin{commC} 
%The reference to \cite[Prop 2.16]{DyHo15} concerns reflection subgroups in general for which the results is not as easy as for standard parabolic subgroups. As Prop 2.16 is not either proved in our paper with Matthew, I thought to give a proof here would be nice 
%\end{commC}
%\begin{commN}
%	Does the general reflection subgroup statement not follow from conjugating to a standard parabolic subgroup?
%\end{commN}
%\begin{commC} No, not all reflection subgroups of $W$ are parabolic subgroups. Already in $B_2$, the subgroup generated by $s,tst$ is not conjugated to a standard parabolic subgroup.
%\end{commC}

\begin{cor}\label{cor:InvSubg} Let $I\subseteq S$ and $w=w_Iw^I$ with $w_I\in W_I$ and $w^I\in X_I$, then $N(w_I)=N(w)\cap \Phi_I^+$. 
\end{cor}
\begin{proof} The left-to-right inclusion follows from Proposition~\ref{prop:Inv}(a) since $w_I$ is a prefix of $w$. Now let $\alpha\in N(w)\cap \Phi^+_I$. So $w_I^{-1}(\alpha)\in \Phi_I$ and $w^{-1}(\alpha)=(w^I)^{-1}w_I^{-1}(\alpha)$ is an element of $\Phi^-$. Now $(w^I)^{-1}(\Phi_I^+)\subseteq \Phi^+$, since  $w^I\in X_I$. Thus $w_I^{-1}(\alpha)$ must be in $\Phi^-_I$, and so $\alpha\in N(w_I)$.
\end{proof}

%%%%%%%%%%%
%\subsection{Cosets of standard parabolic subgroups and morphisms of automata} 
%\label{sse:cosets}
%$\ $ 

%%%%%%%%%%%%%%%%%%%%
\subsection{Parabolic Subgroups}\label{sse:Parabolic}

We now discuss the behaviour of Garside shadows with respect to standard parabolic subgroups.

Let $\B$ be a Garside shadow in $(W,S)$ and $W_I$ be the standard parabolic subgroup generated by $I\subseteq S$.  Then $\B\cap W_I$ is a Garside shadow~\cite[Remark~2.5(c)]{DyHo15}.  Let $\mathcal A^{(I)}_\B(W,S)$ be the restriction of the automaton $\mathcal A_\B(W,S)$ to the states corresponding to $\B\cap W_I$ and to the transitions corresponding to $s\in I$. 

\begin{prop} Let $\B$ be a Garside shadow and $I\subseteq S$.
\begin{enumerate}[(a)] 
\item The restriction of $\pi_\B$ to $W_I$ is the  $(B\cap W_I)$-projection  $\pi_{\B\cap W_I}:W_I\to \B\cap W_I$.
\item $\mathcal A^{(I)}_\B(W,S)=\mathcal A_{\B\cap W_I}(W_I,I)$.
\end{enumerate}
\label{prop:parabolic_restriction}
\end{prop}
\begin{proof} (a) By definition of Garside shadow projections, we  need to show that for any $w\in W_I$ we have
$
\{g\in \B\,|\, g\leq_R w\} = \{g\in \B\cap W_I \,|\, g\leq_R w\}.
$

The right-to-left inclusion is obvious. Now let $w\in W_I$ and $g\leq_R w$. Since $w\in W_I$, any reduced word for $w$  uses only  letters  from  $I$, by~\cite[Corollary~1.4.8(ii)]{BjBr05}.  In particular any prefix of $w$ is in $W_I$. Since $g\leq_R w$, $g$ is a prefix of $W_I$. Therefore $g\in W_I$, which concludes the proof of (a).

\smallskip
\noindent (b) By definition, the states of $\mathcal A^{(I)}_\B(W,S)$ and of $\mathcal A_{\B\cap W_I}(W_I,I)$ are the same. The fact that the transitions are the same follows by (a).
\end{proof}

\begin{remark}[Minimal automata and restriction to standard parabolic subgroups]\label{rem:MinParab}  We  do not know if the restriction of a Garside shadow to  $W_I$ is compatible with  Garside closure~\cite[Remark~2.5(c)]{DyHo15}. In other words,  if $\B$ is a Garside shadow in $(W_I,I)$,   is  $\Gar_S(\B)\cap W_I=\B$? In particular, we do not know if  $\mathcal A_{\tilde S}^{(I)}(W,S)=\mathcal A_{\tilde I}(W_I,I)$. 
\end{remark}

 Another way to restrict a Garside shadow to a standard parabolic subgroup is by the mean of the minimal coset representatives decomposition; the associated automaton structure is discussed in Proposition~\ref{prop:ParabSurj}. Recall that any element $w\in W$ has a unique decomposition $w=w_Iw^I$ with $w_I\in W_I$ and $w^I\in X_I$.  We denote by $p_I:W\to W_I$ the projection defined by $p_I(w):=w_I$.

\begin{prop}\label{prop:ProjParabolic}  Let $B$ be a Garside shadow in $(W,S)$ and $I\subseteq S$. 
\begin{enumerate}[(a)]
\item The set $p_I(B)$ is a Garside shadow in $(W_I,I)$.
\item We have $p_I\circ \pi_B= \pi_{p_I(B)}\circ p_I$.
%\item The projection $p_I$  induces  a totally   surjective morphism from $\mathcal A_B(W,I)$ to $\mathcal A_{p_I(B)}(W_I,I)$.
\end{enumerate}
\end{prop}

%\NEWP
%\COMMC{I put Philippe's counterexample in a remark. I rephrase one thing in order for the definition of $W$ not to be broken on two line, once the OLD removed}
\begin{remark}
 We have $B\cap W_I\subseteq p_I(B)$, but equality does not hold in general. Indeed, let $S=\{s,t,u\}$ and $W=\{S\,|\ s^2=t^2=u^2=1,\ su=us\}.$  One checks that $B:=\{1,s,t,u,su,tu,stu\}$ is a Garside shadow for $(W,S)$.
Now pick $I=\{s,t\}$. Then we have $p_I(stu)=st\notin B$ while $stu\in B$, so $st\in p_I(B)\setminus (B\cap W_I)$.
\end{remark}

To prove the proposition, we need the following lemma.

\begin{lem}\label{lem:ParabJoin} Let $X$ be a bounded set in $W$ and $I\subseteq S$, then $p_I(\bigvee X)= \bigvee p_I(X)$.
\end{lem}
\begin{proof} By Corollary~\ref{cor:InvSubg} and Proposition~\ref{prop:Inv}(b) we have
$$
N\left(p_I\left(\bigvee X\right)\right)= N\left(\bigvee X\right)\cap \Phi_I= \cone_{\Phi}\left(\bigcup_{x\in X} N(x)\right)\cap \Phi_I.
$$
Since our statement is about combinatorics of reduced words, we consider without loss of generality  the simple system to be a basis of $V$.  So in particular  $\Span(\Phi_I)$ is a supporting hyperplane of $\cone(\Delta)$. Therefore, 
$$
\cone_{\Phi}\left(\bigcup_{x\in X} N(x)\right)\cap \Phi_I= \cone_{\Phi}\left(\bigcup_{x\in X} N(x)\cap \Phi_I\right),
$$
since  there are only finitely many generators for each cone.  So by Corollary~\ref{cor:InvSubg}  we obtain
$$
N\left(p_I\left(\bigvee X\right)\right)=  \cone_{\Phi_I}\left(\bigcup_{x\in X} N(p_I(x))\right).
$$
Finally, by Proposition~\ref{prop:Inv}(b) and Corollary~\ref{cor:InvSubg} again we have:%\COMMP{I put the computation on one line}
\[
N\left(\bigvee p_I(X)\right)=\cone_{\Phi_I}\left (\bigcup_{x\in X} N(p_I(x))\right)
= N\left(p_I\left(\bigvee X\right)\right).\qedhere
\]
\end{proof}

\begin{proof}[Proof of Proposition~\ref{prop:ProjParabolic}] (a) We verify the conditions in Definition~\ref{def:garside_shadow}. It is clear that $p_I(B)\subseteq W_I$. Now, since $p_I(s)=s$ for any $s\in I$ and $I\subseteq S\subseteq B$ we have $I\subseteq p_I(B)$.  For Condition~(i), consider $X\subseteq B$ bounded in $W$.  So $p_I(X)$ is bounded in $W$, so its join $\bigvee p_I(X)$ exists. We have to show that $\bigvee p_I(X)\in p_I(B)$. By Lemma~\ref{lem:ParabJoin} we have $\bigvee p_I(X)= p_I(\bigvee X)$, which is an element of $p_I(B)$ since  $\bigvee X\in B$. For Condition~(ii), consider $w\in B$ and a suffix $v$ of $w_I$. Since $w^I\in X_I$, the expression $vw_I$ is reduced. Therefore, $vw^I$ is a suffix of $w\in B$. Since $B$ is a Garside shadow, $vw^I\in B$. Furthermore, $p_I(vw^I)=v$, so $v\in p_I(B)$.

\smallskip

\noindent (b)  It is enough to show that for $w\in W$, we have
$$
\{p_I(g)\,|\, g\in B,\ g\leq_R w\} = \{g'\in p_I(B) \,|\, g'\leq_R p_I(w)\}.
$$
But this follows easily from Proposition~\ref{prop:Inv}(a) together with Corollary~\ref{cor:InvSubg}.

\smallskip

\end{proof}

%%%%%%%%%%%%%%%%%%%%%%%%%%%%%%%
\section{Morphisms and Garside Shadow Automata}\label{se:Can}
%%%%%%%%%%%%%%%%%%%%%%%%%%%%%%%

In this section, we discuss  morphisms between Garside shadow automata, then we  compare the automata of a particular family of Garside shadows, the set of $n$-low elements with  the family of $n$-canonical automata. We first recall the definitions of morphisms of automata, minimal automata, and  the concept of minimal roots for $(W,S)$.

%%%%%%%%%%%
\subsection{Morphisms of automata} \label{se:Morphism}

We refer the reader to~\cite[Chapter II(3)]{Sa09} for  additional details on morphisms of automata. Here we shall only use this notion in a particular case suited to our various automata.

\begin{defi}[{see~\cite[Chapter II(3)]{Sa09}}]
\label{def:morphism}
    Let $\mathcal A=(Q,q_0,F,\delta)$ and $\mathcal
A'=(Q',q'_0,F',\delta')$ be two  finite deterministic automata over
the same alphabet $S$.  A function $f:Q\to Q'$ is a  {\em morphism of
automata} between $\mathcal A$ and $\mathcal A'$ if
\begin{enumerate}[(i)]
\item $f(q_0)=q_0'$;
\item  $f(F)\subseteq F'$;
\item If $q_1 \overset{s}{\rightarrow} q_2$  is a transition  in
$\mathcal A$, then $f(q_1) \overset{s}{\rightarrow} f(q_2)$ is a
transition in $\mathcal A'$.
\end{enumerate}

 A morphism of automata $f$ is \emph{totally surjective} if $f$ is surjective, satisfies $f^{-1}(F')=F$ and if, for any transition
$q'_1\overset{s}{\rightarrow}q'_2$ in $\mathcal A'$, there exists
$q_1,q_2$ such that  $f(q_1)=q'_1,f(q_2)=q'_2$ and
$q_1\overset{s}{\rightarrow}q_2$ in $\mathcal A$. In this case $\mathcal A'$ is called a {\em  quotient} of $\mathcal{A}$.
\end{defi}

If $f$ is a morphism between $\mathcal A$ and $\mathcal A'$ then
$\mathcal{L}(\mathcal{A})\subseteq \mathcal{L}(\mathcal{A'})$. If $f$
is totally surjective then
$\mathcal{L}(\mathcal{A})=\mathcal{L}(\mathcal{A'})$.

 The following proposition gives a first example of a totally surjective morphism related to Garside shadow automata and arising from the surjection $p_I$ from  Proposition~\ref{prop:ProjParabolic}.

 \begin{prop}\label{prop:ParabSurj}
 Let $I\subseteq S$ and $B$ be a Garside shadow in $(W,S)$.  The automaton $\mathcal A_B (W,I)$  is defined by taking the same states  $B$, initial state $e$, and final states  as $\mathcal A_B(W,S)$, but with only the transitions of $\mathcal A_B(W,S)$ corresponding to letters in $I$.  Then the surjection $p_I:B\to p_I(B)$ induces a totally surjective morphism from $\mathcal A_B(W,I)$ to $\mathcal A_{p_I(B)}(W_I,I)$.
 \end{prop}
 \begin{proof}
Let us first show that $p_I$ verifies  the conditions in Definition~\ref{def:morphism}. We have $p_I(e)=e$, and all states are final in both $\mathcal A_B(W,I)$ and $\mathcal A_{p_I(B)}(W_I,I)$, so (i) and (ii) hold. To prove (iii), let $w \overset{s}{\rightarrow} \pi_B(sw)$ be a transition in $\mathcal A_B(W,I)$ with $s\in I\setminus D_L(w)$. We have to show that $p_I(w) \overset{s}{\rightarrow} p_I(\pi_B(sw))$ is a transition in $\mathcal A_{p_I(B)}(W_I,I)$.  Since $D_L(w)\cap I=D_L(p_I(w))$, there is a transition  $p_I(w) \overset{s}{\rightarrow} \pi_{p_I(B)}(sp_I(w))$.  The equality $p_I(\pi_B(sw))=\pi_{p_I(B)}(sp_I(w))$ is then  guaranteed by Proposition~\ref{prop:ProjParabolic}(b) since $p_I(sw)=sp_I(w)$ for any $s\in I\setminus D_L(w)$.
		
 To prove that $p_I$ is totally surjective, let $p_I(w)\overset{s}{\rightarrow}\pi_{p_I(B)}(sp_I(w))$ be a transition in $\mathcal A_{p_I(B)}(W_I,I)$, with $w\in B$ and $s\in I\setminus D_L(p_I(w))$. Then $w \overset{s}{\rightarrow} \pi_B(sw)$ is a transition in $\mathcal A_B(W,I)$ since $D_L(w)\cap I=D_L(p_I(w))$,     and we conclude as in the previous paragraph.
\end{proof}

%%%%%%%%%%%
\noindent\textbf{Minimal automata}. 
Given a regular language $L\in S^*$, there exists an automaton $\mathcal{R}(L)$ which recognizes $L$ and is a quotient of all automata that recognize $L$, called the {\em minimal automaton of $L$}. 

It can be constructed as follows: given $u\in S^*$, define
$u^{-1}L$ to be the set of $v\in S^*$ such that $uv\in L$, and let
$Q_L=\{u^{-1}L~|~u\in S^*\}$. We then define
$\mathcal{R}(L)=(Q_L,q_L,F_L,\delta_L)$ with $q_L=L$,
$F_L=\{u^{-1}L~|~u\in L\}$ and transitions
$\delta_L(u^{-1}L,a)=(ua)^{-1}L$.  This automaton  clearly recognizes $L$.

%\COMMP{this is where the notion of complete automaton is necessary}
Now pick any deterministic, complete\footnote{This means that $\delta$ is defined everywhere. Any automaton can be transformed into a complete one by adding a non-final {\em sink state $\dagger$} and transitions $\delta(q,s):=\dagger$ whenever $\delta$ is not previously defined.} automaton $\mathcal{A}$ such that
$L=\mathcal{L}(\mathcal{A})$. Given $q\in Q$, let $L_q(\mathcal{A})$
be the language recognized by the automaton $\mathcal{A}$ with $q$
replacing $q_0$ as initial state. If $L_q(\mathcal{A})=L_{q'}(\mathcal{A})$ 
then $q$ and $q'$ are called \emph{equivalent} states. Then $q\mapsto
L_q(\mathcal{A})$ is a totally surjective morphism from $\mathcal{A}$
to $\mathcal{R}(L)$. Therefore in order to prove that an automaton is minimal, one must show that distinct states are never equivalent. %It follows that $\mathcal{R}(L)$ is the unique deterministic complete\COMMP{here is where complete comes up... quotient defined ?} automaton with the minimal number of states that recognizes $L$; it is called the {\em minimal automaton} recognizing $L$  and is the smallest quotient of $\mathcal A$\WEN.
 %The cardinality of $Q_L$ is called the {\em state complexity} of the language $L$.

\begin{remark} Denote by $\mathcal A_{min}(W,S)$ the minimal automaton that recognizes the language $\Red(W,S)$. For $I\subseteq S$, we define the automaton $\mathcal A^{(I)}_{min}(W,S)$ to be the restriction of $\mathcal A_{min}(W,S)$ to the transitions in $I$ and the states that can be reached from the initial state using these transitions. We now show that $\mathcal A^{(I)}_{min}(W,S)=\mathcal A_{min}(W_I,I)$, so that minimal automata remain minimal upon restriction to a parabolic subgroup.

Let $q_1,q_2$ be  distinct states in $\mathcal A^{(I)}_{min}(W,S)$; $q_1,q_2$ can be reached by reading (reduced words for elements) $w_1,w_2\in W_I$, respectively. Since $q_1,q_2$ are non-equivalent states in   $\mathcal A_{min}(W,S)$  by minimality, there exists $w\in W$ such that $w_1w$ is reduced while $w_2w$ is not reduced.  Now use the decomposition $w=w_Iw^I$ with $w_I\in W_I$ and $w^I\in X_I$, then $w_1w_I$ is reduced while $w_2w_I$ is not reduced. So the states are not equivalent in  $\mathcal A^{(I)}_{min}(W,S)$, which is therefore the minimal  automaton $\mathcal A_{min}(W_I,I)$ that recognizes $\Red(W_I,I)$.

Therefore---assuming Conjecture~\ref{conj:1}---we conclude that $\mathcal A_{\tilde S}^{(I)}(W,S)=\mathcal A_{\tilde I}(W_I,I)$, resolving our question in Remark~\ref{rem:MinParab}.
\end{remark}

%%%%%%%%%%%
\subsection{Inclusion of Garside shadows and morphisms of automata} 
\label{sse:GarsideAutMorph}

\begin{prop}\label{prop:Compo}  If $C\subseteq B$ are  two Garside shadows, then $\pi_C\circ \pi_B=\pi_C$. 
\end{prop}
\begin{proof} It is enough to show that for $w\in W$, we have
$$
\{g\in C\,|\, g\leq_R \pi_B(w)\} = \{g\in C \,|\, g\leq_R w\}.
$$
The left-to-right inclusion follows from Proposition~\ref{prop:Proj1}(b). Now let $g\in C$ such that $g\leq_R w$. Since $C\subseteq B$ we have by Proposition~\ref{prop:Proj1}(b,c) that 
$$g=\pi_B(g)\leq_R \pi_B(w),$$
which concludes the proof.
\end{proof}

\begin{cor}\label{cor:ShadowSurj}  If $C\subseteq B$ are two Garside shadows, then  the $C$-projection $\pi_C$  induces  a totally surjective morphism from $\mathcal A_B$ to $\mathcal A_C$. In particular,  $\mathcal A_{\tilde S}$ is a quotient of any Garside shadow automaton.
\end{cor}

\begin{proof} 
The $C$-projection $\pi_C:B\to C$ is surjective, since if $w\in C$, then $w\in B$ and $\pi_C(w)=w$.  We now show that $\pi_C$ verifies  the conditions in Definition~\ref{def:morphism}. 
	\begin{enumerate}[(i)]
		\item Since $e \in C$, $\pi_C(e)=e$.
		\item Let $w \overset{s}{\rightarrow} \pi_B(sw)$ be a transition in $\mathcal A_B$ with $s\notin D_L(w)$. We have to show that $\pi_C(w) \overset{s}{\rightarrow} \pi_C(\pi_B(sw))=\pi_C(sw)$ is a transition in $\mathcal A_C$, using Proposition~\ref{prop:Compo}.  Since $D_L(w)=D_L(\pi_C(w) )$ by Proposition~\ref{prop:Proj2}, we have $s\notin D_L(\pi_C(w) )$. So $\pi_C(w)\overset{s}{\rightarrow} \pi_C(s\pi_C(w))$ is a transition in $\mathcal A_C$ by definition, and we conclude by Proposition~\ref{prop:Proj3}.
		\item This holds since all states are final in both automata and $\pi_C$ is surjective.
\end{enumerate}

To prove that $\pi_C$ is totally surjective, it remains to show that,  if $v\overset{s}{\rightarrow}v'$ is a transition in $\mathcal A_C$, there is a transition $u \overset{s}{\rightarrow}u'$ in $\mathcal A_B$ with $\pi_C(u)=v$ and $\pi_C(u')=v'$.  But this is guaranteed by taking $u:=v$ and $u':=v$ since $C\subseteq B$.
\end{proof}

To conclude this discussion, we show Conjecture~\ref{conj:1} in the finite case. 

\begin{prop}\label{prop:FiniteMin} Assume that $W$ is finite. Then $\tilde S=W$ and $\mathcal A_{\tilde S}$ is the minimal automaton that recognizes $\Red(W,S)$. 
\end{prop}

\begin{proof} The fact that $W$ is finite implies  that  $\tilde S=W$~\cite[Proposition 2.2(3)]{DyHo15}. Therefore $\pi_{\tilde S}=\pi_W=Id_W$, the identity map on $W$. Thus $\mathcal A_{\tilde S}$ has states indexed by $W$ and transitions $w\overset{s}{\rightarrow} sw$ if $s\notin D_L(w)$.

Let $u,v\in W$ be two equivalent states
, i.e.  for any $s_1,s_2,\cdots,s_k$, we have that $s_k\cdots s_1 u$ is reduced if and only if $s_k\cdots s_1 v$ is reduced. We must prove that $u=v$ . Let $k\geq 0$ be maximal so that there exist $s_1,s_2,\cdots,s_k$ with $s_k\cdots s_1 u$ reduced; note that $k$ exists since $W$ is finite. By hypothesis, $s_k\cdots s_1 v$ is reduced and $k$ is  necessarily  also maximal for that property. But there is a unique element $w$ satisfying $D_L(w)=S$, namely the longest element $w_\circ$~\cite[Proposition 2.3.1(ii)]{BjBr05}. This shows that $s_k\cdots s_1 u=w_\circ=s_k\cdots s_1 v$ and thus $u=v$.
\end{proof}

%%%%%%%%%%%%%%%%%%%%%%
\subsection{Small Inversion Sets and Low Elements} \smallskip

In~\cite{BrHo93}, the authors introduced a partial order $\preceq$ on $\Phi^+$ called the {\em dominance order }  defined by:
$$
\alpha\preceq \beta \quad \iff \quad (\forall w\in W, \beta \in N(w)\implies \alpha \in N(w)).
$$
The $\infty$-depth of $\beta\in\Phi^+$ is the number of positive roots strictly dominated by $\beta$:  
$$
\dep_\infty(\beta):=|\{\alpha\in \Phi^+\,|\, \alpha\prec \beta \}|.
$$

%\OLD $\ $
%%%%%%%%%%%%%%%
%\subsubsection{Small inversion sets}
%$\ $\DLO 

\begin{defi}\label{def:nsmall}
Let $n\in \mathbb N$, we say that a root $\beta\in \Phi^+$ is {\em $n$-small} if $\dep_\infty(\beta)\leq n$ and set $\Sigma_n(W)$  to be  the set of $n$-small roots.
\end{defi}
 A $0$-small root is  called a {\em small root} and we write  $\Sigma(W):=\Sigma_0(W)$.  B.~Brink and R.~Howlett showed in~\cite{BrHo93} that $\Sigma_n$ is finite for $n=0$. This result was later extended by  X.~Fu~\cite{Fu12} to all $n \geq 0$. The {\em (left) $n$-small inversion set of $w\in W$} is 
$$
\Sigma_n(w):=N(w)\cap \Sigma_n,
$$
and we denote by $\Lambda_n(W)\subseteq \mathcal P(\Sigma_n(W))$ the set of all $n$-inversion sets. Since $\Sigma_n(W)$ is finite, $\Lambda_n(W)$ is also finite.   We write $\Sigma(w):=\Sigma_0(w)$.

%\OLD $\ $
%%%%%%%%%%%%%%%%%%%
%\subsubsection{Low elements}
%$\ $\DLO 

\begin{defi}\label{def:nlow}
An element $w\in W$ is {\em $n$-low} if $N(w)=\cone_\Phi(\Sigma_n(w))$. We denote by $L_n(W)$ the set of $n$-low elements in $W$.
\end{defi}
A $0$-low element is  called a {\em low element} and we write  $L(W):=L_0(W)$. Low elements were introduced by P.~Dehornoy, M.~Dyer, and the first author in~\cite{DDH14}, and extended for any $n\in \mathbb N$ by  M.~Dyer and the first author in~\cite{DyHo15}. We refer the reader to \cite[\S3.1-\S3.3]{DyHo15} for more details and examples; examples of low elements are also given in Figure~\ref{fig:a2_c2_shi_arrangements}. We summarize  here some results concerning $n$-small inversion sets and $n$-low elements.

\begin{thm}[\cite{DyHo15}]\label{thm:Low} Let $n\in\mathbb N$.
\begin{enumerate}[(a)]
\item The map $\Sigma_n: L_n(W)\to \Lambda_n(W)$ is injective.
\item The set $L_n(W)$ of $n$-low elements is finite and closed under join in $(W,\leq_R)$.
\item The set of low elements $L(W)$ is a finite Garside shadow in $(W,S)$.
\item If $(W,S)$ is finite, affine or with Coxeter graph with edges labelled by $3,\infty$, then the set $L_n(W)$ of $n$-low elements is a finite Garside shadow in $(W,S)$.
\end{enumerate}
\end{thm}

The statements (a) and (b) are~\cite[Proposition 3.26]{DyHo15}, the statement (c) is  \cite[Theorem 1.1]{DyHo15} and (d) is~\cite[Theorem 1.3 and Theorem 4.17]{DyHo15}. We end this discussion by recalling two conjectures from~\cite{DyHo15}:

%\COMMN{Make this a conjecture in this paper as well, for later reference?}\COMMC{Not necessarily, do we refer to it often?}
\begin{description}
\item[{\cite[Conjecture~1]{DyHo15}}] The set $L_n(W)$ is a finite Garside shadow in $(W,S)$.
\item[{\cite[Conjecture~2]{DyHo15}}] The map $\Sigma_n: L_n(W)\to \Lambda_n(W)$ is a bijection.
\end{description}

%%%%%%%%%%%%%%%%%%%%%%
\subsection{Low Element Automata and Canonical Automata} Let $n\in\mathbb N^*$, the {\em $n$-canonical automaton} is the finite automaton $\mathcal A_n(W,S)$ over $S$ defined as follows:
\begin{itemize}
\item the   (finite)   set of states  is $\Lambda_n(W)$;
\item the initial state is $\emptyset(=\Sigma_n(e))$ and all states are final;
\item the transitions are: $A\overset{s}{\rightarrow} \{\alpha_s\}\cup \left( s(A)\cap\Sigma_n\right)$ whenever   $\alpha_s\notin A$.
\end{itemize}

 As shown in ~\cite{DyHo15}, if  $A=\Sigma_n(w)$ then  $s\notin D_L(w)$ if and only if   $\alpha_s\notin A$, and in this case $\{\alpha_s\}\cup \left( s(A)\cap\Sigma_n\right)=\Sigma_n(sw)$. The transitions are thus well defined. Also, one has immediately that if $w=s_1\cdots s_k$ is reduced, then the path from $\emptyset$ with labels $s_1,\ldots,s_k$ ends in the state $\Sigma_n(w)$.

Therefore the $n$-canonical automaton $\mathcal A_n(W,S)$ recognizes $\Red(W,S)$, for any $n\in\mathbb N$.

The $0$-canonical automaton, or simply  the  {\em canonical automaton}, was studied by  H.~Eriksson in his thesis~\cite{Er94} and named in~\cite[\S4.8]{BjBr05}.

 When  $L_n(W)$ is a Garside shadow in $(W,S)$---which we suspect is {\em always} the case~\cite[Conjecture~1]{DyHo15}--- we may  consider the associated finite Garside shadow  projection and  automaton. 

\begin{prop}\label{prop:CanLow} Let $n\in\mathbb N$. 
\begin{enumerate}[(a)]
\item The Garside shadow projection $\pi_{L_n(W)}$ is well-defined.
\item The map $\pi_n:\Lambda_n(W)\to L_n(W)$ defined by $\pi_n(\Sigma_n(w)):=\pi_{L_n(W)}(w)$ is a well-defined surjection. 
\item The map $\pi_0$ induces a totally surjective morphism from the canonical automaton $\mathcal A_0(W,S)$ to the automaton $\mathcal A_{L(W)}(W,S)$.
\item  If  $L_n(W,S)$ is a Garside shadow in $(W,S)$, then $\pi_n$ induces a totally surjective morphism from the $n$-canonical automaton $\mathcal A_n(W,S)$ to the automaton $\mathcal A_{L_n(W)}(W,S)$.
\end{enumerate}
\end{prop}
\begin{proof} 
	\begin{enumerate}[(a)]
	\item  As observed in Remark~\ref{rem:proj_uses},  the definition of the Garside shadow projection  $\pi_{L_n}(W)$ only requires $L_n(W)$ to contain $S$ and be closed under  taking joins, which is guaranteed by Theorem~\ref{thm:Low}(b). 

 	\item The fact that $\pi_n$ is surjective follows from  the definition of $\Lambda_n(W)$.  To  prove that $\pi_n$ is well-defined, let  $u,v\in W$ such that $\Sigma_n(u)=\Sigma_n(v)$. We have to show that $\pi_{L_n(W)}(u)=\pi_{L_n(W)}(v)$. By Definition~\ref{def:garside_projection}, it is enough to  show that 
$$
\{g\in L_n(W)\,|\, g\leq_R u\} = \{g\in L_n(W)\,|\, g\leq_R v\}.
$$
Let $g\in L_n(W)$ such that $g\leq_R u$.  By  Proposition~\ref{prop:Inv}, we have $N(g)\subseteq N(u)$, and therefore $\Sigma_n(g)\subseteq \Sigma_n(u)=\Sigma_n(v)$. Since $g$ is $n$-low we have by definition 
$$
N(g)=\cone_\Phi(\Sigma_n(g))\subseteq \cone_\Phi(\Sigma_n(v))\subseteq N(v).
$$
Therefore, again by Proposition~\ref{prop:Inv}, $g\leq_R v$. This shows the left-to-right inclusion, and we conclude the other inclusion by symmetry.

	\item This follows from (d) and Theorem~\ref{thm:Low}(c).

	\item  We must check the three conditions of Definition~\ref{def:morphism}:
	\begin{enumerate}[(i)]
		\item This follows from the fact that $\pi_n(\emptyset)=\pi_n(\Sigma_n(e))=\pi_{L_n(W)}(e)=e$.
	   \item This follows since every state in both automata is final and $\pi_n$ is surjective.
	   \item By definition of the transitions in $\mathcal A_{L_n(W)}(W,S)$ and $\mathcal A_n(W,S)$, one must check that if $s\notin D_L(w)$, then  
			$\pi_{L_n(W)}(s\pi_{L_n(W)}(w))=\pi_{L_n(W)}(sw). $ This follows immediately from  Proposition~\ref{prop:Proj3}.
\end{enumerate}

To prove that   $\pi_n$ is totally surjective, it remains to show that,  if $w\overset{s}{\rightarrow}\pi_{L_n(W)}(sw)$ is a transition in $\mathcal A_{L_n(W)}(W,S)$, then  $\Sigma_n(w) \overset{s}{\rightarrow} \Sigma_n(sw)$ is a transition in $\mathcal A_n(W,S)$.  But  $s\notin D_L(w)$, therefore $\alpha_s\notin \Sigma_n(w)$,  and so $\Sigma_n(w) \overset{s}{\rightarrow} \Sigma_n(sw)$ is a transition  in $\mathcal A_n(W,S)$.  

\end{enumerate}
\end{proof}

\begin{remark}\label{rem:Isom} Let $n\in \mathbb N$. If the map $\Sigma_n: L_n(W)\to \Lambda_n(W)$ is a bijection,  i.e. {\cite[Conjecture~2]{DyHo15}} has a positive answer, then  $\pi_n$ would induce a  {\em isomorphism} from the $n$-canonical automaton  $\mathcal A_n(W,S)$ to $\mathcal A_{L_n(W)}(W,S)$. 
\end{remark}

\begin{prop}
	If $W$ is finite, then the canonical automaton $\mathcal{A}_{0}(W,S)$ is minimal.
	\label{prop:CanFinite}
\end{prop}
\begin{proof}
	Since $W$ is finite, we have $\Phi^+=\Sigma(W)$. Then in particular $\Sigma(w)=N(w)$ for any $w\in W$ and therefore $L(W)=W=\tilde S$, by Proposition~\ref{prop:FiniteMin}. So $N=\Sigma:L(W)=W=\tilde S\to \Lambda(W)$ is a bijection and therefore $\mathcal A_0(W,S)$ and $\mathcal A_{\tilde S}(W,S)$ are isomorphic. The result follows  therefore  by Proposition~\ref{prop:FiniteMin}.
\end{proof}

The next corollary of Proposition~\ref{prop:CanLow}, together with Corollary~\ref{cor:ShadowSurj}, strengthens the evidence for  Conjecture~\ref{conj:1}. 

\begin{cor}\label{cor:MinSurj} The automaton $\mathcal A_{\tilde S}(W,S)$ associated to the smallest Garside shadow $\tilde S$ in $(W,S)$ is a quotient of all canonical automata $\mathcal A_n(W,S)$.
\end{cor}
 
Figure~\ref{fig:diagram_automata}  illustrates all of our automata recognizing $\Red(W,S)$, and the maps between them.

\begin{figure}[!ht]
\includegraphics[width=0.7\textwidth]{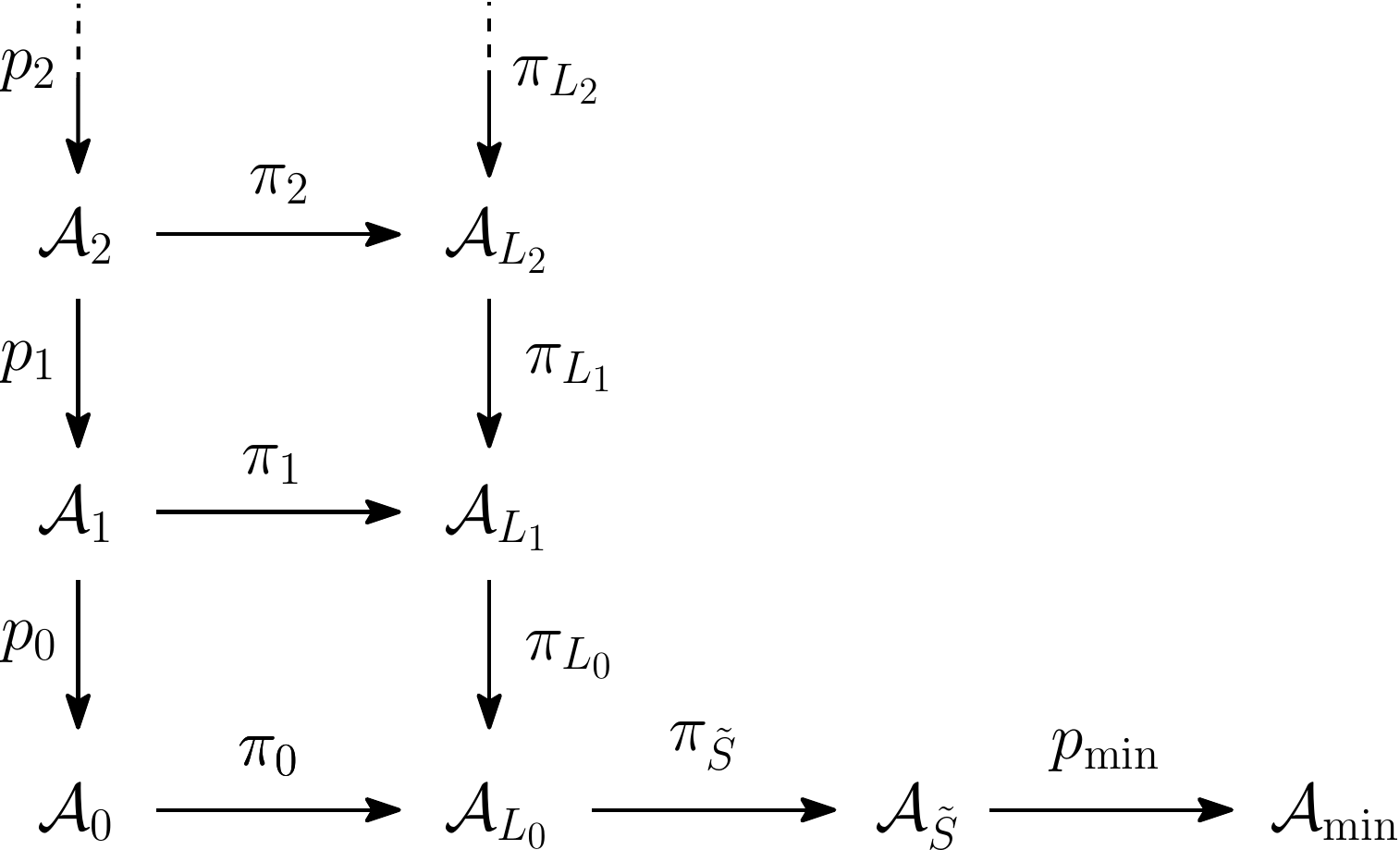}
\caption{Commutative diagram relating the various automata described in this article. The projections $\pi_i$ are conjectured to be bijections~\cite[Conjecture~2]{DyHo15}, and so is $p_{\min}$ (Conjecture\ref{conj:1}). The sets $L_i\subseteq W$ are conjectured to be Garside shadows~\cite[Conjecture~1]{DyHo15}; only $L_0$ is known to be. Finally Conjecture~\ref{conj:2} characterizes when the bottom row $p_{\min}\circ \pi_{\tilde S}\circ \pi_0$ is an isomorphism.
\label{fig:diagram_automata}}
\end{figure}

%%%%%%%%%%%%%%%%%%%%%%%%%%%%%%%
\subsection{Minimality of the canonical automaton}\label{se:Min}

A positive root $\beta\in \Phi^+=\cone_\Phi(\Delta)$ has a unique expression with nonnegative linear combination of vectors in $\Delta$: $\beta=\sum_{s\in S} a_s\alpha_s$, with $a_s\geq 0$; we define the  {\em support} of $\beta$  to be the set
\[\supp(\beta):=\{s\in S\,|\, a_s>0\}.\]

We say that a positive root $\beta$ is {\em spherical} if the standard  parabolic subgroup $W_{\supp(\beta)}$ is finite, and we write $\Phi^+_{sph}$ for the set of spherical roots. 

 Spherical roots are always small. Now if the reverse inclusion holds, the following proposition shows that the canonical automaton is minimal, so that one implication in Conjecture~\ref{conj:2} is true.

\begin{prop}
\label{prop:spherical_implies_minimaL}
	Let $W$ be irreducible.  If $\Sigma=\Phi^+_{sph}$, then $\mathcal{A}_{0}(W,S)$ is minimal.
\end{prop}

The following proof is inspired by Theorem V.8 in P.~Headley's thesis~\cite{He94}.

\begin{proof}
Let $\Sigma(u)$ and $\Sigma(v)$ be two equivalent states of $\mathcal A_0(W,S)$. This means that for any $s_1,s_2,\cdots,s_k$ in $S$, we have that $s_k\cdots s_1 u$ is reduced if and only if $s_k\cdots s_1 v$ is reduced. We have to show that $\Sigma(u)=\Sigma(v)$.
	
	Assume that they are distinct, so that, up to exchanging the role of $u$ and $v$,  there is $\alpha\in \Sigma(u)\setminus \Sigma(v)$. By assumption, $\Sigma=\Phi^+_{sph}$, so there is $I\subseteq S$ such that $W_I$ is finite and $\alpha\in \Phi_I^+$.

	 Now we use the decompositions $u=u_Iu^I$ and $v=v_Iv^I$  in $W_I\times X_I$.  The expression $wu$  is reduced if and only if $wu_I$ is reduced, since $gu^I$ is reduced for any $g\in W_I$. So we have that for any $s_1,\ldots,s_k\in I$, $s_k\cdots s_1 u_I$ is reduced if and only if $s_k\cdots s_1 v_I$ is reduced.
	
		  Since $W_I$ is finite,  the automaton $\mathcal A_0(W_I,I)$ is minimal by Proposition~\ref{prop:CanFinite}. Therefore $\Sigma(u_I)=\Sigma(v_I)$.   Note that $\Sigma(u_I):=\Sigma(u)\cap \Phi_I$ and $\Sigma(v_I):=\Sigma(v)\cap \Phi_I$ are small inversion sets for $(W_I,I)$, by Corollary~\ref{cor:InvSubg} and the definition of small inversion sets. But $\alpha$ was chosen to be in $\Phi_I\cap \Sigma(u)$, which contradicts that $\alpha\in \Sigma(u)\setminus \Sigma(v)$.  Therefore, $\Sigma(u)=\Sigma(v)$.
	\end{proof}
 We  conclude by proving Theorem~\ref{thm:CanMin}. 

\begin{proof}[Proof of Theorem~\ref{thm:CanMin}]
\begin{enumerate}
 \item This is Proposition~\ref{prop:CanFinite}, since $\mathcal{A}_{0}(W,S)$ and $\mathcal{A}_{\tilde S}(W,S)$ are isomorphic in this case.
 
\item Here $\Sigma=\Phi^+_{sph}=S$; see for instance~\cite[Proposition 5.1(iii)]{DDH14}. So we are in the case of Proposition~\ref{prop:spherical_implies_minimaL}.

\item In this case $\Phi^+_{sph}$ consists of the union of all $\Phi^+_{s,t}$ where $m_{st}<\infty$. This is equal to the whole of$\Sigma$ since the support of a small root is a tree with no $\infty$-edge~\cite{Br98}. We conclude again by Proposition~\ref{prop:spherical_implies_minimaL}.
Note that one can actually give an explicit description of the canonical automaton in this case and prove its minimality directly.

\item  The fact that the automaton is minimal in this case  is due to Eriksson~\cite[Theorem~80]{Er94}.  Now recall that the Coxeter graph  is a simply-laced cycle. Since the support of a small root is a tree~\cite{Br98}, we have $\Sigma=\Phi^+_{sph}$ here and the conjecture holds by Proposition~\ref{prop:spherical_implies_minimaL}.

\item The case of complete graphs was already checked, so one may assume that we have generators $s,t,u$ with $m_{su}=2$ and $3\leq m_{st}\leq m_{tu}$. Denote $m=m_{st}$ and $p=m_{tu}$. If $p=\infty$, or if $m=3$ and $p<6$, we have $\Sigma=\Phi^+_{sph}$, so Proposition~\ref{prop:spherical_implies_minimaL} gives us the result.

 We may now assume ($m=3$ and $p\geq 6$) or ($m,p\geq 4$); in particular $W$ is not finite. Write $c_i=2\cos(\pi/i)$. Then $\alpha:=us\alpha_t=c_m\alpha_s+\alpha_t+c_p\alpha_u$ is a small root which is not spherical, so that $\Sigma\neq\Phi^+_{sph}$.  To show that the conjecture holds we thus need to find two distinct equivalent states in $\mathcal{A}_{0}(W,S)$.

 Now $su$ and $tsu$ are reduced words, with distinct final states in $\mathcal{A}_{0}(W,S)$ given by $\Sigma(su)=\{\alpha_s,\alpha_u\}$ and $\Sigma(sut)=\{\alpha_s,\alpha_u,\alpha\}$.  We have $D_L(su)=D_L(tsu)=\{s,u\}$, so only $t$ can be read from any of these states. Now a quick computation shows $t(\alpha)=\alpha+(c_m^2+c_p^2-1)\alpha_t$. Since $c_m^2+c_p^2-1\geq 2$ for all considered values of $m$ and $p$, we have $t(\alpha)\notin\Sigma$ and therefore $\Sigma(tsu)=\Sigma(tsut)$. This shows that $\Sigma(su)$ and $\Sigma(tsu)$ are equivalent states, and thus that $\mathcal{A}_{0}(W,S)$ is not minimal.

\end{enumerate}
\end{proof}

\begin{remark}[On Conjectures~\ref{conj:1} and~\ref{conj:2}]\label{rem:Sage}
 Using Sage~\cite{Sage-Combinat,sage}, we wrote code to compute the set of small roots.   We used these to compute the canonical automaton, from which we determined the minimal automaton.  It is simple to test if a given small roots is spherical by examining the simple roots that occur in its support, from which we are able to check Conjecture~\ref{conj:2}. This code is sufficiently fast to compute examples in rank 5---for example, we determined that the minimal automaton for $\widetilde{D}_5$ has size $58965$.  
 
 We also wrote a naive implementation to determine the minimal Garside shadow using Definition~\ref{def:garside_shadow} to check Conjecture~\ref{conj:1}.  This code finishes in a few minutes on standard hardware in rank four (and below), but already takes longer than several hours in rank five.

Our software confirms that Conjectures~\ref{conj:1} and~\ref{conj:2} hold for all Coxeter groups of rank $4$ with edge labels less than $10$.  Figure~\ref{fig:evidence} includes data for a few selected Coxeter groups of low rank.

\begin{figure}[htbp]
	\[\begin{array}{|c|c|c|c|c|c|c|} \hline
		\text{Name} & \text{Coxeter Diagram} & |A_0(W,S)| & |A_{\widetilde{S}}(W,S)| & |A_{min}(W,S)| & |\Sigma| & |\Phi^+_{sph}| \\ \hline
		\widetilde{A}_2 & \huge\raisebox{\height}{\scalebox{0.5}{$\xymatrix{1 \ar@{-}[rr]|{3} \ar@{-}[dr]|{3} & & 2 \ar@{-}[dl]|{3}  \\ & 3 &}$}}\normalsize & 16 & 16 & 16 & 6 & 6\\ \hline
		\widetilde{C}_2 & \huge\raisebox{\height}{\scalebox{0.5}{$\xymatrix{1 \ar@{-}[r]|{4} & 2 \ar@{-}[r]|{4} & 3}$}}\normalsize & 25 & 24 & 24 & 8 & 7\\ \hline
		\widetilde{G}_2 & \huge\raisebox{\height}{\scalebox{0.5}{$\xymatrix{1 \ar@{-}[r]|{3} & 2 \ar@{-}[r]|{6} & 3}$}}\normalsize & 49 & 41 & 41 & 12 & 8\\ \hline
		\widetilde{A}_3 &  \huge\raisebox{\height}{\scalebox{0.5}{$\xymatrix{1 \ar@{-}[r]|{3} \ar@{-}[d]|{3} & 2 \ar@{-}[d]|{3}  \\ 4\ar@{-}[r]|{3}  & 3}$}}\normalsize & 125 & 125 & 125 & 12 & 12\\ \hline
		\widetilde{C}_3 & \huge\raisebox{\height}{\scalebox{0.5}{$\xymatrix{1 \ar@{-}[r]|{4} & 2 \ar@{-}[r]|{3} & 3 \ar@{-}[r]|{4} & 4}$}}\normalsize & 343 & 317 & 317 & 18 & 15\\ \hline
		\widetilde{B}_3 & \huge\raisebox{\height}{\scalebox{0.5}{$\xymatrix{  &  & 3 \\
 1 \ar@{-}[r]|{4}  & 2 \ar@{-}[dr]|{3} \ar@{-}[ur]|{3} &  \\
 & & 4}$}}\normalsize & 343 & 315 & 315 & 18 & 15\\ \hline
%		\widetilde{A}_4 & \huge\raisebox{\height}{\scalebox{0.5}{$\xymatrix{& 1 \ar@{-}[r]|{3} \ar@{-}[ld]|{3} & 2 \ar@{-}[dd]|{3}  \\ 5\ar@{-}[rd]|{3} & & \\ & 4\ar@{-}[r]|{3}  & 3}$}}\normalsize & 1296 & 1296 & 1296 & 20 & 20\\  \hline
%		\widetilde{C}_4 & \huge\raisebox{\height}{\scalebox{0.5}{$\xymatrix{1 \ar@{-}[r]|{4} & 2 \ar@{-}[r]|{3} & 3 \ar@{-}[r]|{3} & 4 \ar@{-}[r]|{4} & 5}$}}\normalsize & 6561 & 5860 & ? & 32 & 26 \\ \hline
%		\widetilde{B}_4 & \huge\raisebox{\height}{\scalebox{0.5}{$\xymatrix{ &  &  & 4 \\
% 1 \ar@{-}[r]|{4}& 2 \ar@{-}[r]|{3}  & 3 \ar@{-}[dr]|{3} \ar@{-}[ur]|{3} &  \\
% & & & 5}$}}\normalsize & 6561 & 5789 & ? & 32 & 26 \\ \hline
%		\widetilde{D}_4 & \huge\raisebox{\height}{\scalebox{0.5}{$\xymatrix{ 2 &  & 3 \\
%  & 1 \ar@{-}[dr]|{3} \ar@{-}[ur]|{3} \ar@{-}[dl]|{3} \ar@{-}[ul]|{3} &  \\
%4 & & 5}$}}\normalsize & 2401 & 2400 & ? & 24 & 23 \\ \hline
%		\widetilde{F}_4 & \huge\raisebox{\height}{\scalebox{0.5}{$\xymatrix{ 1 \ar@{-}[r]|{3} &2 \ar@{-}[r]|{4}  &3 \ar@{-}[r]|{3} &4 \ar@{-}[r]|{3}   &5}$}}\normalsize& 28561 & 22428 & ? & 48 & 31 \\ \hline
%		& \huge\raisebox{\height}{\scalebox{0.5}{$\xymatrix{  &  & 3\ar@{-}[dd]|{5} \\
% 1 \ar@{-}[r]|{4}  & 2 \ar@{-}[dr]|{3} \ar@{-}[ur]|{3} &  \\
% & & 4}$}}\normalsize & 68 & 68 & 68 & 12 & 12 \\ \hline
		& \huge\raisebox{\height}{\scalebox{0.5}{$\xymatrix{  &  & 3\ar@{-}[dd]|{5} \\
 1 \ar@{-}[r]|{4}  & 2 \ar@{-}[dr]|{4} \ar@{-}[ur]|{3} &  \\
 & & 4}$}}\normalsize & 92 & 92 & 92 & 15 & 15\\ \hline
		& \huge\raisebox{\height}{\scalebox{0.5}{$\xymatrix{  &  & 3\ar@{-}[dd]|{5} \\
 1 \ar@{-}[r]|{4}  & 2 \ar@{-}[dr]|{5} \ar@{-}[ur]|{3} &  \\
 & & 4}$}}\normalsize & 164 & 164 & 164 & 21 & 21\\ \hline
		& \huge\raisebox{\height}{\scalebox{0.5}{$\xymatrix{  &  & 3\ar@{-}[dd]|{5} \\
 1 \ar@{-}[r]|{4}  & 2 \ar@{-}[dr]|{6} \ar@{-}[ur]|{3} &  \\
 & & 4}$}}\normalsize & 91 & 80 & 80 & 18 & 14\\ \hline
		& \huge\raisebox{\height}{\scalebox{0.5}{$\xymatrix{1 \ar@{-}[r]|{3} \ar@{-}[d]|{5} & 2 \ar@{-}[d]|{6} \ar@{-}[dl]|{9}  \\ 4\ar@{-}[r]|{8}  & 3}$}}\normalsize & 100 & 90 & 90 & 30 & 25\\ \hline
%		& [9,8,9,8,9,7] & 99 & 99 & 99 & 42 & 42 \\ \hline
	\end{array}\]
\caption{Numerical data for selected Coxeter groups.  Note that in affine type $\widetilde{W}_n$, $|A_0(W,S)|=(h+1)^n$ and $|\Sigma|=nh$, where $h$ is the Coxeter number of the corresponding finite Weyl group.}
\label{fig:evidence}
\end{figure}
\end{remark}

%\NEWN
%\subsection{Further evidence for Conjecture~\ref{conj:2}}
%\label{se:Evidence}
% In \S\ref{se:Evidence}, we show that Conjecture~\ref{conj:2} holds for all Coxeter groups of rank at most four and edge labels at most nine.

%Following the advice of an anonymous referee, we give additional evidence 

%\WEN

%%%%%%%%%%%%%%%%%%%%%%%%%%%%%%%%%
\subsection{Canonical automata and Shi arrangements}\label{sse:Shi}

We end this article by describing some rank $3$ examples of automata. It turns out these examples can be drawn in a very nice way: their states form a convex set in the (dual of the) geometric representation of $(W,S)$. The reason in the affine case is related to a  property of the Shi arrangement, which leads  us to discuss a generalization of the Shi arrangement for any Coxeter system.

Let $\Phi_0$ be a reduced, irreducible, crystallographic root system of rank $r$ for a finite Weyl group $W_0$ in a real vector space $V_0$ with $W_0$-invariant positive definite scalar product $\langle\cdot,\cdot \rangle$.  Let $\Phi_0^+$ be a choice of positive roots, let $\Delta_0=\{\alpha_1,\alpha_2,\ldots,\alpha_n\}$ be the corresponding simple roots. The {\em height} of a positive root $\alpha = \sum_{i=0}^n c_i \alpha_i$ is $\sum_{i=0}^n c_i$; for example, the {\em highest root} $\alpha_h$ {\em in $\Phi_0^+$} has height $h-1$, where $h$ is the Coxeter number of $W_0$.  Define $V:=V_0\oplus\mathbb{R}\delta$ and define the set of affine roots to be
\[\Phi:=\{\alpha+k\delta \,|\, \alpha\in\Phi\text{ and }k\in\mathbb{Z}\}.\] 
 The positive affine roots are 
$\Phi^+:=\{\alpha+k\delta\,|\, \alpha\in\Phi_0^+\text{ and }k\geq0 \in \mathbb{Z}\}\cup\{\alpha+k\delta \,|\, \alpha\in-\Phi_0^+\text{ and }k>0\in \mathbb{Z}\}$, and the simple affine roots are
$\Delta:=\Delta_0\cup\{\alpha_0\},$ where $\alpha_0:=-\alpha_h+\delta$.  

For $\alpha\in\Phi_0$ and $k\in\mathbb{Z}$,  we consider in $V_0$, seen as an affine space, the affine hyperplane 
\[H_{\alpha,k}:=\{x\in V_0\,|\,\langle x,\alpha\rangle=k\}.\]  

%\begin{commC}
%Above, the scalar product $\langle\ \rangle$ is not defined on $V$ but only on $V_0$. I made modifications about that in the textbecause I think you are here looking at $V_0$ as an affine space, and not at $V$. Also, usually the affine Coxeter arrangement, and therefore the Shi arrangement,  leaves in a space of the same dimension than $V_0$. Since it is not the way I define it usually (I use the dual and the quotient space of $V$ by the radical as in Humphreys), please check my modifications properly. 
%
% The heights of roots was not defined at the right place.
%\end{commC}
 The affine Weyl group $W$  is  the group generated by affine reflections in the simple affine hyperplanes : $H_{\alpha,0}$ for $\alpha\in \Delta_0$ and $H_{\alpha_h,1}$.  The fundamental alcove $\mathcal K$ is the (interior of the) compact region bounded by the simple affine hyperplanes.  The closure of $\mathcal K$ is a fundamental domain for the action of $W$.

The {\em $n$-Shi arrangement} is the collection of hyperplanes \[\Shi_n(W) := \left\{ H_{\alpha,k} \,|\, \alpha \in \Phi^+, -n+1\leq k \leq n\right\}.\]  
We abbreviate $\Shi(W):=\Shi_1(W)$, and call it the {\em Shi arrangement}\footnote{This was called the {\em sandwich arrangement} in~\cite{He94}}.  The roots corresponding to the hyperplanes in $\Shi_n(W)$ coincide with the $n$-small roots, so that $\Sigma_n$ can be thought of as a generalization of the $n$-Shi arrangement to {\em any} Coxeter group.
%If $w_a\in\wa$, write it as $w_a=wt_{\mu}$ for unique $w\in W$ and $\mu\in\Q$ and define \[w_a(\alpha+k\delta)=w(\alpha)+(k-\langle\mu,\alpha\rangle)\delta.\] 
 The crystallographic affine root system $\Phi$ is easily and bijectively convertible to a root system; so the dominance order and $n$-small roots are well-defined in a crystallographic root system. It particular, the only relations in the dominance order on $\Phi^+$ are $\alpha+k\delta \preceq \alpha+\ell\delta$ for $\alpha \in \Phi, k \leq \ell \in \mathbb{Z}$; see~\cite[Example~3.9]{DyHo15}.  We obtain therefore the following proposition. 

\begin{prop}\label{prop:Shi}  If $(W,S)$ is of affine type, then 
	\[\Shi_n(W) = \{H_\alpha \,|\, \alpha \in \Sigma_n(W)\}.\]
\end{prop}
The Shi arrangements for types $\widetilde{A}_2$ and $\widetilde{C}_2$ are drawn in Figure~\ref{fig:a2_c2_shi_arrangements}.

\begin{figure}
	$\begin{array}{cc}
		\raisebox{-0.5\height}{\includegraphics[width=.5\textwidth]{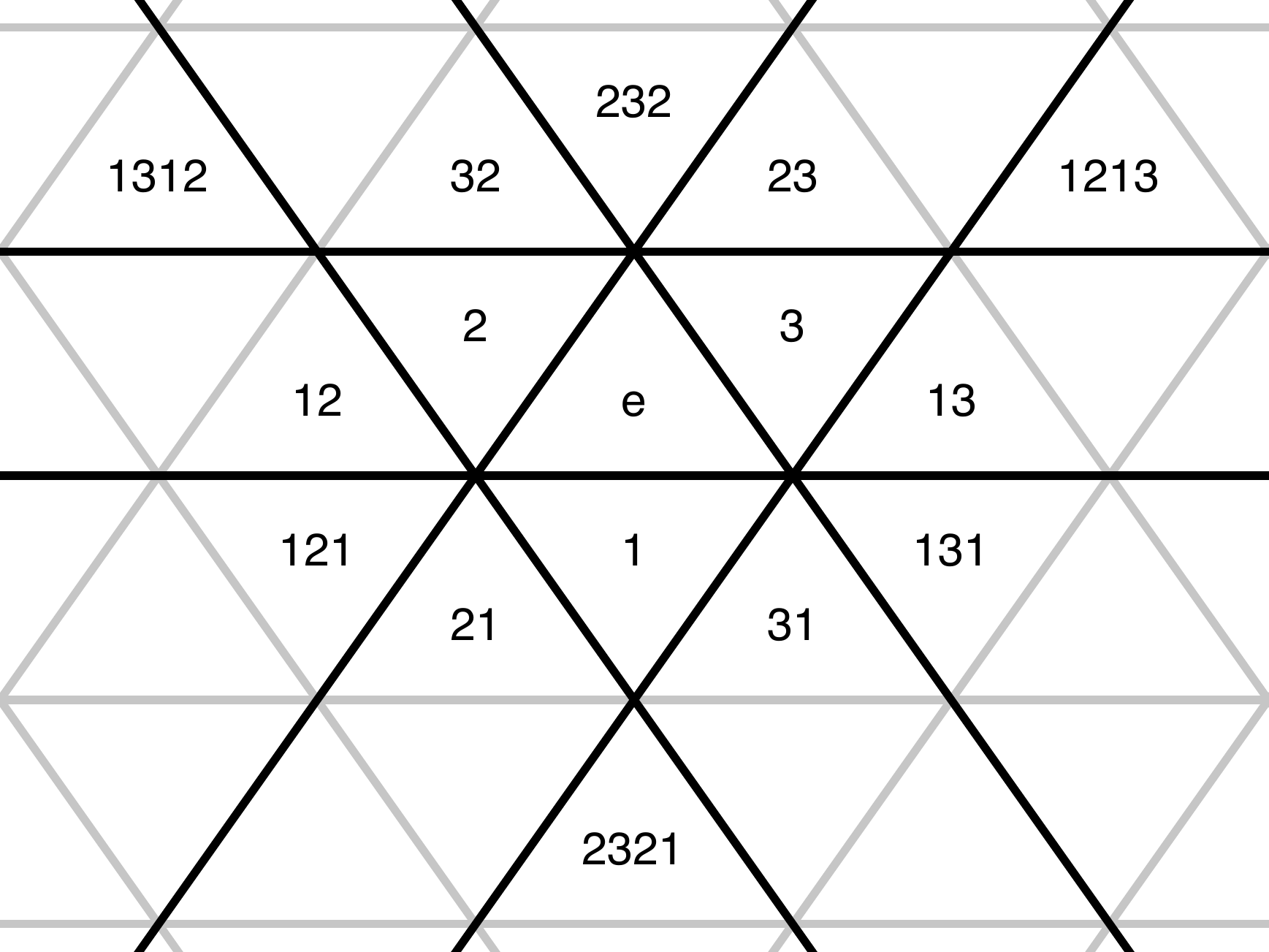}} & \raisebox{-0.5\height}{\includegraphics[width=.5\textwidth]{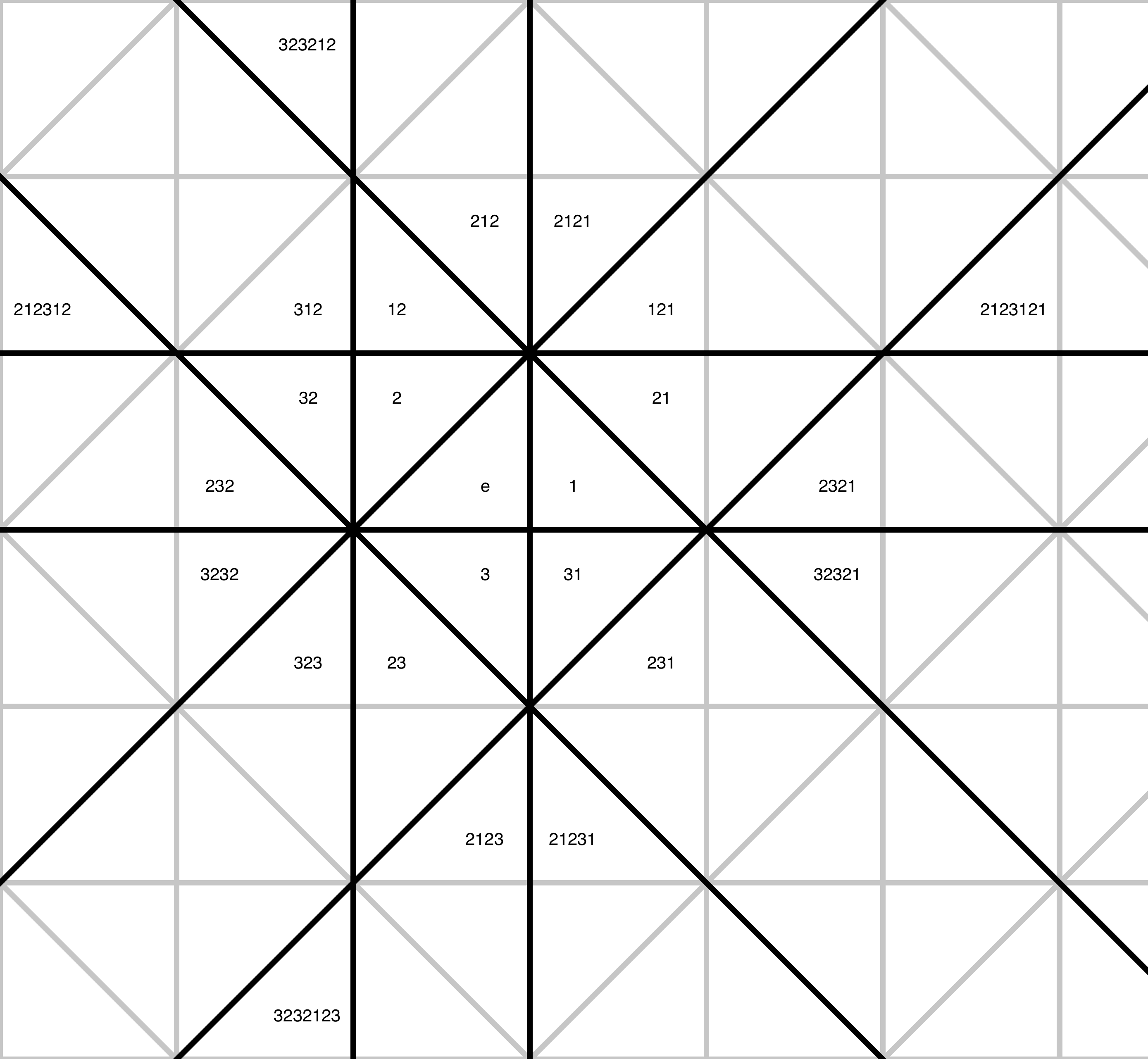}}
	\end{array}$
\caption{$A_2$ and $C_2$ Shi arrangements.}
\label{fig:a2_c2_shi_arrangements}
\end{figure}

In affine type, the small inversion sets have previously been studied under the guise of the minimal alcoves of the $n$-Shi arrangement.  More precisely, for $W$ of affine type,
	$\Lambda_n(W) = \{w \in W  \,|\, \{ w(\alpha_s)  \,|\, s \in D_L(w)\} \subseteq \Sigma_n\}$.  The corresponding statement for $n$-low elements and general type is given as~\cite[Conjecture~2]{DyHo15}, restated above in~\S\ref{se:Can}.

It turns out that there are $(nh+1)^r$ $n$-low elements in affine type.  The reason for this is that the inverses of such elements coalesce into an $(nh+1)$-fold dilation of the fundamental alcove.

\begin{thm}[J.~Y.~Shi]
	Let $\mathcal{K}$ be the fundamental alcove for an affine Weyl group $W$, and let $h$ be the Coxeter number of the corresponding finite Weyl group $W_0$.  Then \[\{w^{-1}\mathcal{K} \,|\,  w \text{ an $n$-low element} \} \cong (nh+1)\mathcal{K}.\]
\label{thm:shi_regions}
\end{thm}

In particular, the alcoves corresponding to the inverses of $n$-low elements form a convex set.  This theorem is illustrated for the infinite dihedral group $\widetilde{A}_1 = I_2(\infty)$ in Figures~\ref{fig:InfiniteDi} and~\ref{fig:a1_automaton}, which show the automata built from the $1$- and $2$-low elements, respectively.  Figure~\ref{fig:a2_c2_automata} illustrates this theorem for types $\widetilde{A}_2$ and $\widetilde{C}_2$, simultaneously drawing the automaton.

We note that convexity does not necessarily hold for the subset of alcoves coming from the inverses of elements in $\tilde{S}$, as seen for example in Figure~\ref{fig:a2_c2_automata}---for $\widetilde{C}_2$, $\tilde{S}= \Sigma \setminus \{s_1s_3s_2\}$.

\begin{figure}[htbp]
		\raisebox{-0.5\height}{\includegraphics[width=.8\textwidth]{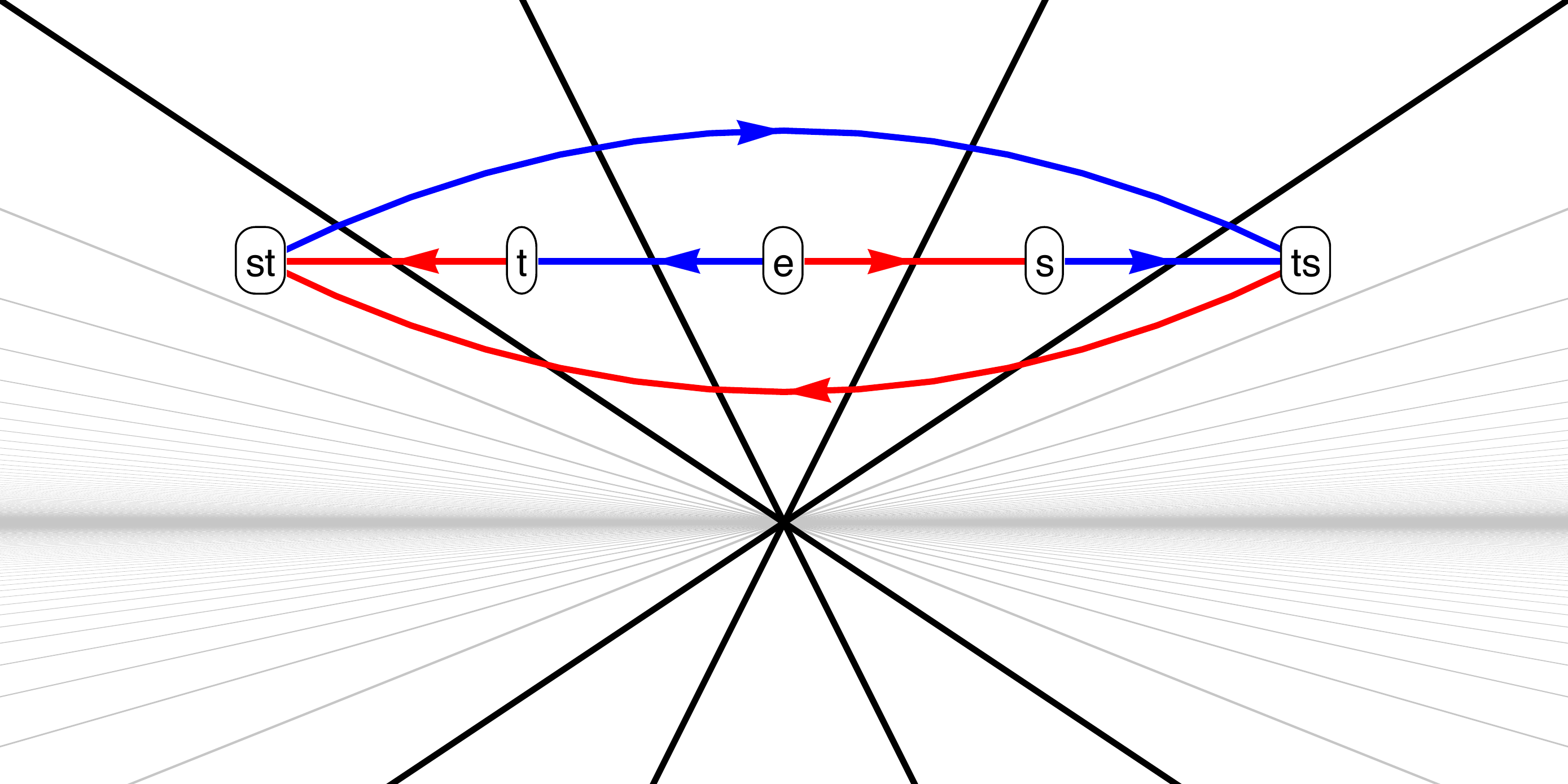}} 
\caption{The automaton $\mathcal{A}_1(I_2(\infty),S)$, drawn using Theorem~\ref{thm:shi_regions}.}
\label{fig:a1_automaton}
\end{figure}

\begin{figure}[htbp]
	$\begin{array}{cc}
		\raisebox{-0.5\height}{\includegraphics[width=.5\textwidth]{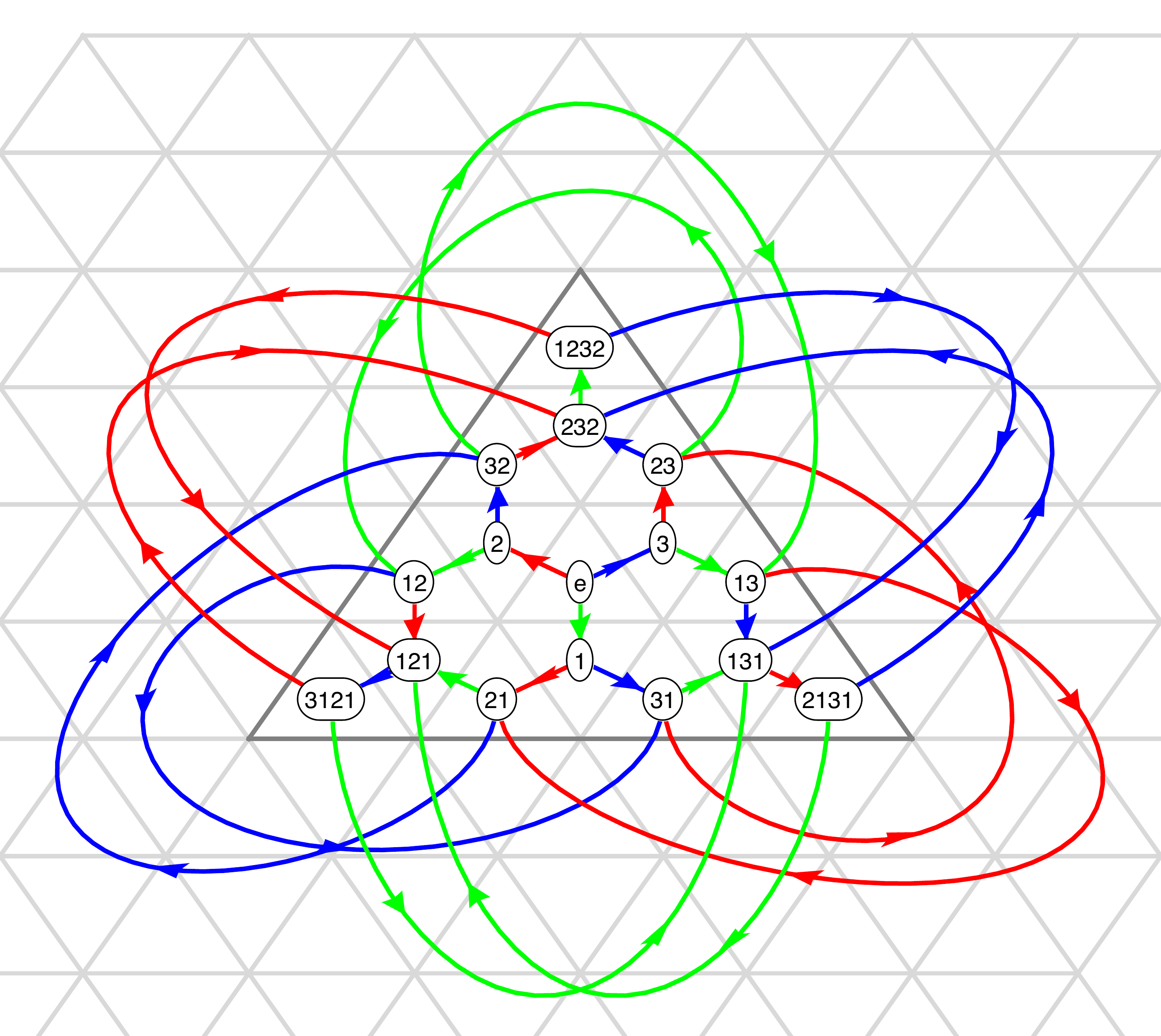}} & \raisebox{-0.5\height}{\includegraphics[width=.5\textwidth]{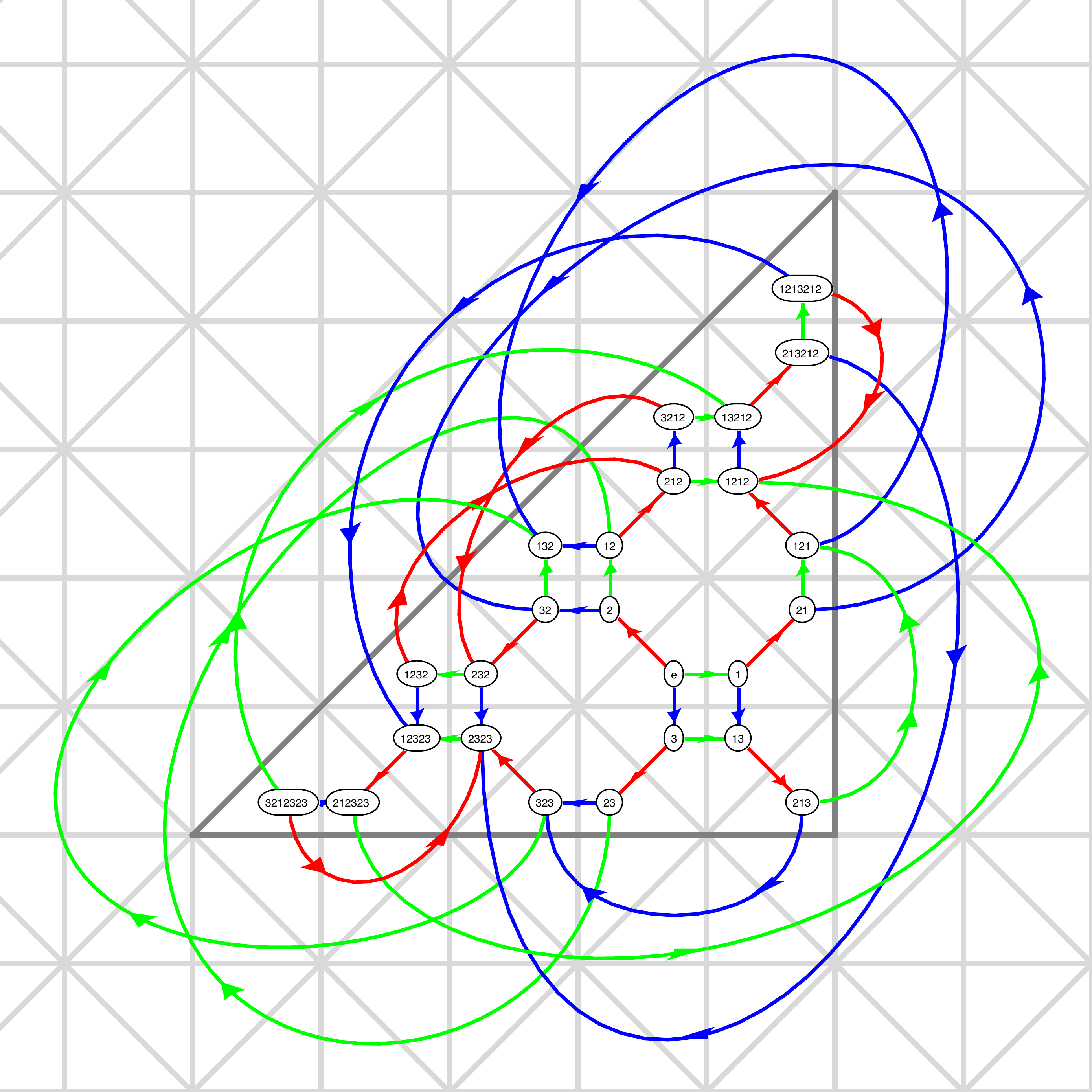}}
	\end{array}$
\caption{The automata $\mathcal{A}_0(\widetilde{A}_2,S)$ and $\mathcal{A}_0(\widetilde{C}_2,S)$, drawn using Theorem~\ref{thm:shi_regions}.  A [green, red, blue] edge represents multiplication by $[s_1, s_2,s_3]$.  There is one omitted (red) edge between $132$ and $213$ in $\mathcal{A}_0(\widetilde{C}_2,S)$.}
\label{fig:a2_c2_automata}
\end{figure}
%\COMMN{I suppressed one edge in the $C_2$ one, because I didn't know how to make it pretty.}

On the basis of the affine rank three examples, it is tempting to conjecture that equivalent states are given by intersecting \emph{intervals} in the weak order with $L_n(W)$.  The (non-affine) triangle group $(5,3,5)$ is a counterexample to this claim.

\medskip

When the Coxeter system $(W,S)$ is of {\em indefinite type}, i.e., $W$ is not finite nor affine, the {\em isotropic cone $Q=\{x\in V\,|\, B(x,x)=0\}$},  and the region where $x\in V$ verifies $Q(x,x)<0$,  are nonempty. In this case, following \cite{HoLaRi14,DyHoRi13}, we consider the {\em projective representation for $(W,S)$} associated to the geometric representation of $(W,S)$, with roots system~$\Phi$ and simple system $\Delta$ in \S\ref{sse:Root}. More precisely,  since $\Phi=\Phi^+\sqcup \Phi^-$ is encoded by the set of positive roots $\Phi^+$, we represent $\Phi$ by an `affine cut' $\h\Phi$: there is an affine hyperplane $V_1$ in $V$ {\em transverse to $\Phi^+$}, i.e., for any $\beta\in \Phi^+$, the ray $\mathbb R^+\beta$ intersects $V_1$ in a unique nonzero point $\h\beta$. So $\mathbb R\beta\cap V_1=\{\h\beta\}$ for any $\beta\in\Phi$. The {\em set of normalized roots}  
$\h\Phi=\{\h\beta\,|\, \beta\in \Phi\}$ is contained in the compact  set $\conv(\h\Delta)$ and therefore admits a set $E$ of accumulation points called {\em the set of limit roots}, which verifies $E\in \h Q$.  The group $W$ acts on $\h\Phi\sqcup E \cup \conv(E)$ componentwise: $w\cdot x=\widehat{w(x)}$.   

Now, the role of the affine space $V_0$  for an affine Coxeter system $(W,S)$  with the tiling obtained by the action of $W$ on the fundamental alcove $\mathcal K$ is replaced for indefinite Coxeter systems by a tiling of the {\em imaginary convex body} $\conv(E)$ by the projective action of $W$ on the non-empty fundamental region
$$
K=\{x\in \conv(\h\Delta)\,|\, B(x,\alpha_s)\geq 0, \ \forall s\in S\}.
$$
Denote $H_{\alpha}=\{x\in V\,|\, B(x,\alpha)=0\}$, then $K$ is the region of $\conv(\h \Delta)$ bounded by the hyperplanes $H_\alpha$. 
This is illustrated in Figure~\ref{fig:aut33infty} and Figure~\ref{fig:conj3}; see also \cite[Figures~2~and~14]{DyHoRi13}. We refer the reader to~\cite{DyHoRi13} for more details.

In view of Proposition~\ref{prop:Shi}, it is natural to give the following definition.

\begin{defi} Let $(W,S)$ be an indefinite Coxeter system. The {\em $n$-Shi arrangement} of $(W,S)$ is the collection of hyperplanes 
\[
 \Shi_n(W,S) := \left\{ H_{\alpha}=\{x\in V\,|\, B(x,\alpha)=0\} \,|\, \alpha \in \Sigma_n(W)\right\}.
 \]  
\end{defi}

If~\cite[Conjecture~2]{DyHo15}, restated above in~\S\ref{se:Can}, is true, it would mean that the set $L_n(W)$ of $n$-low elements parameterized the region of the $n$-Shi arrangement. Furthermore, each region $\Shi_n(W,S)$ would be have a unique minimal-length region of the form $w\cdot K$ with $w\in L_n(W)$. Moreover, we observed in numerous cases in rank $3$ and $4$ the following statement.

\begin{figure}
\raisebox{-0.5\height}{\includegraphics[width=.8\textwidth]{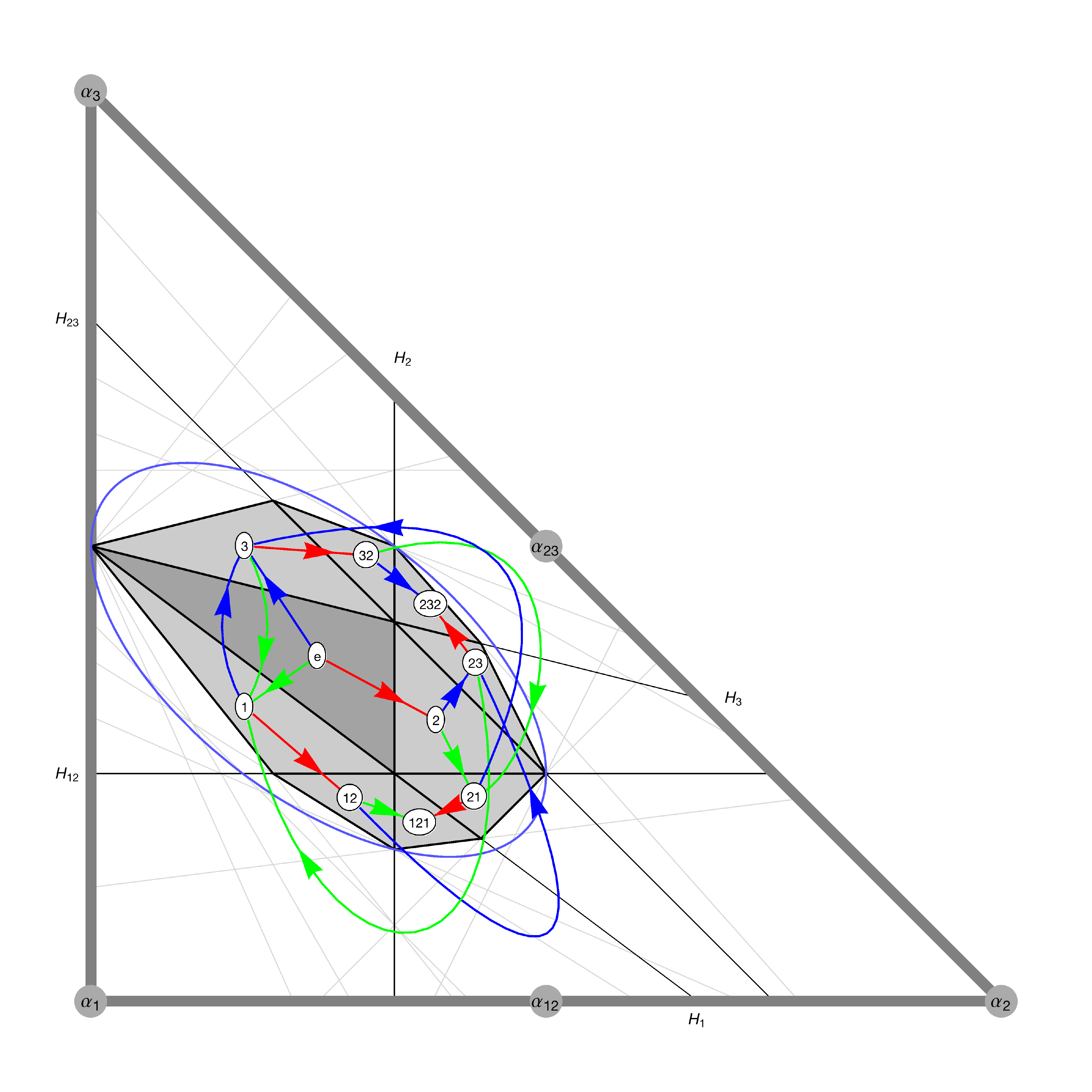}} 
\caption{The automaton $\mathcal{A}_0(W,S)$ for $W$ the triangle group $(3,3,\infty)$.}
\label{fig:aut33infty}
\end{figure}

\begin{conject}\label{conj:3} Let $n\in\mathbb N$, then the subset 
$
\bigcup_{w\in L_n(W)} w^{-1}\cdot K
$
of the imaginary convex body is convex.
\end{conject}

Reasonable evidence for Conjecture~\ref{conj:3} is supplied by the fact that $L_n(W)$ is closed under taking suffixes.  Figure~\ref{fig:conj3} illustrates Conjecture~\ref{conj:3} for several rank three examples.

\begin{figure}[htbp]
	$\begin{array}{ccc}
		\raisebox{-0.5\height}{\includegraphics[width=.3\textwidth]{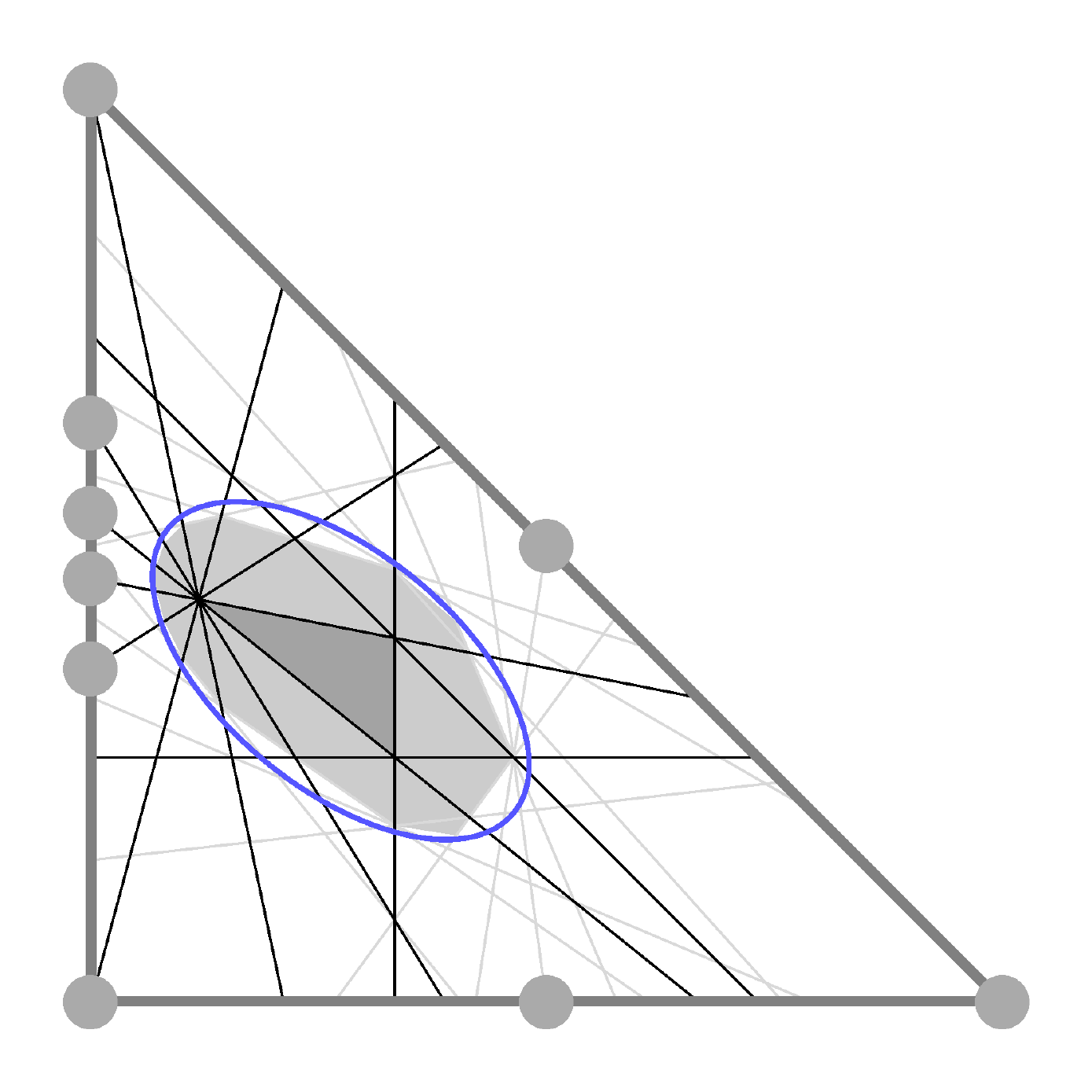}} & \raisebox{-0.5\height}{\includegraphics[width=.3\textwidth]{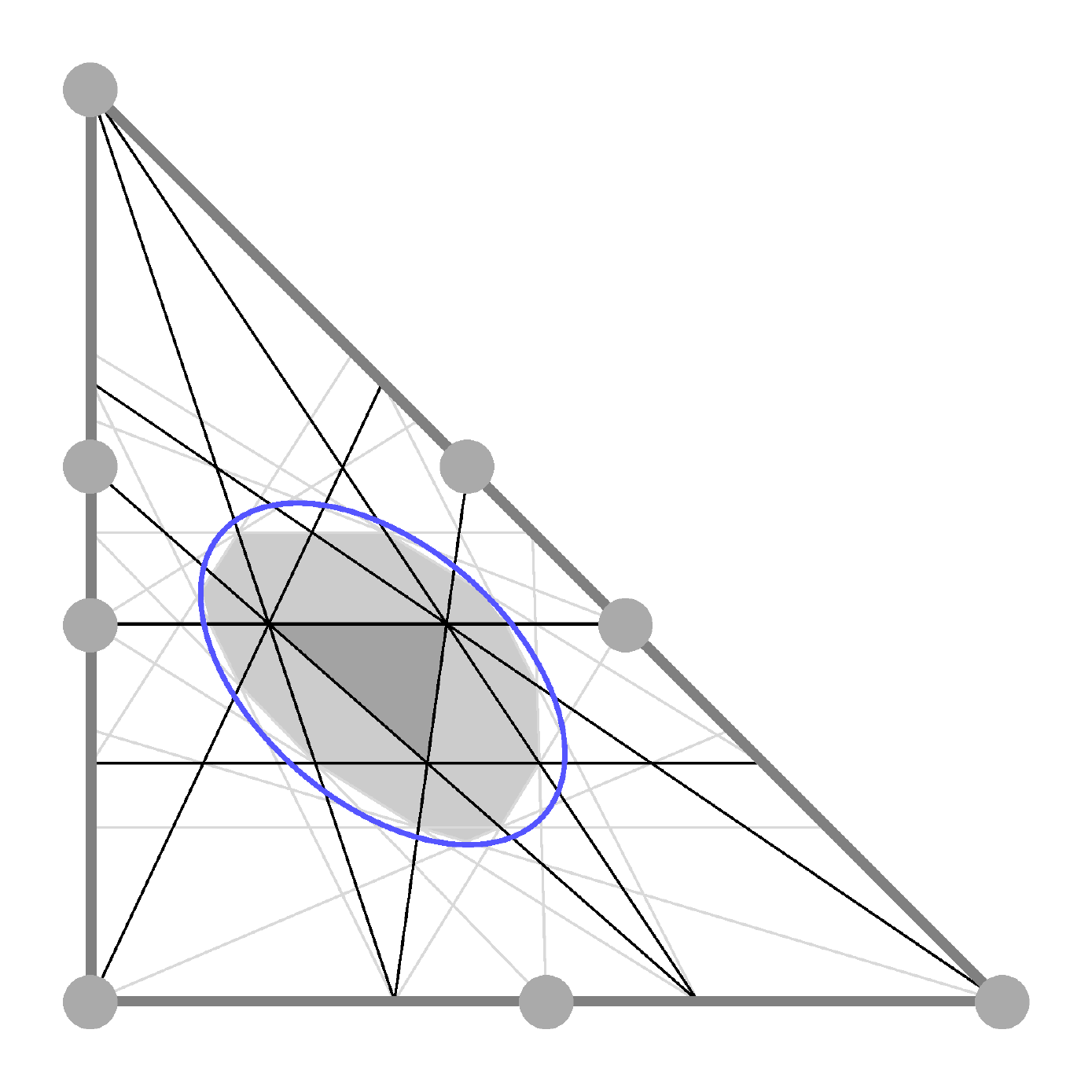}}& \raisebox{-0.5\height}{\includegraphics[width=.3\textwidth]{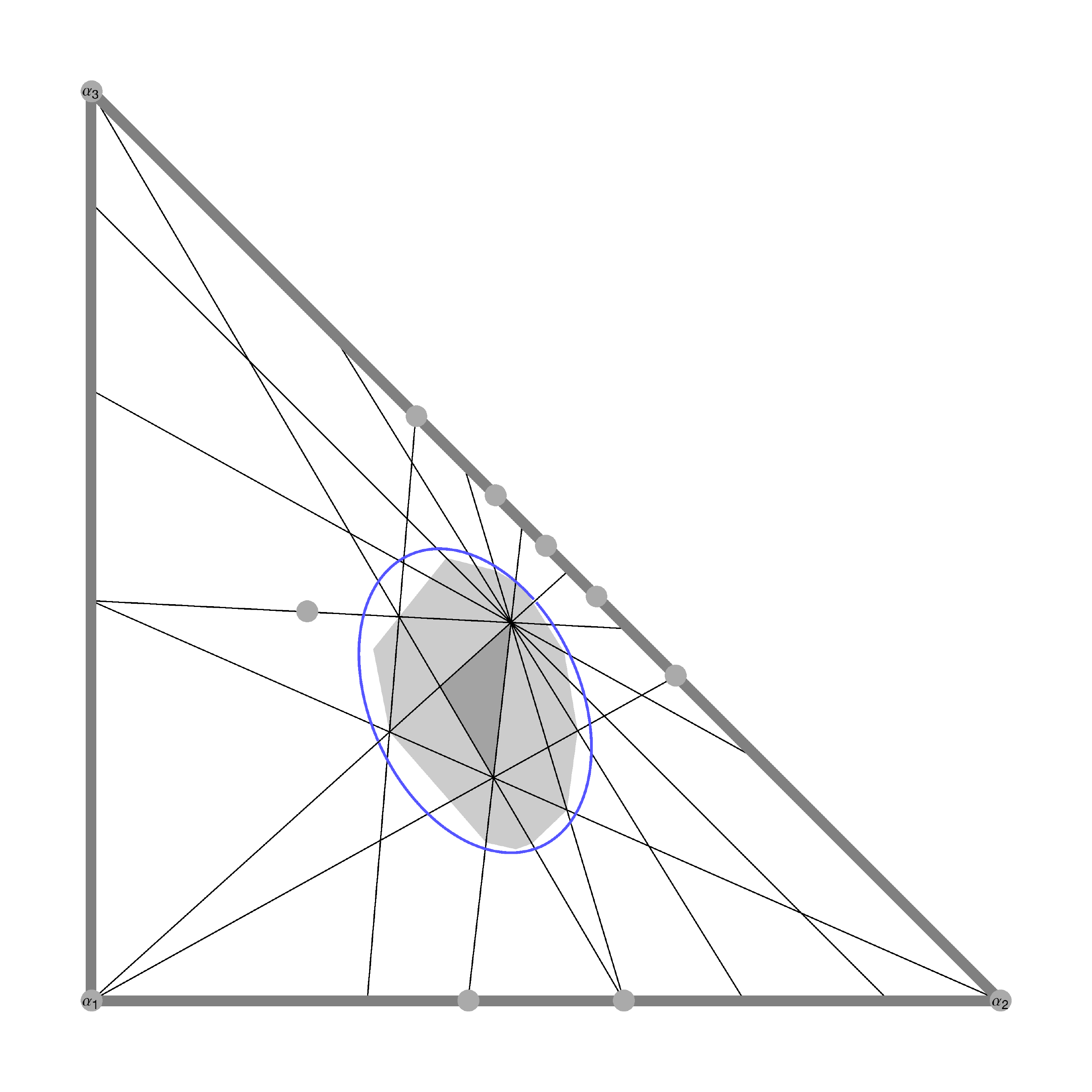}}
	\end{array}$
\caption{The regions $\bigcup_{w\in L_n(W)} w^{-1}\cdot K$ for the triangle groups $(3,3,6)$,$(3,4,4)$, and $(4,7,2)$.}
\label{fig:conj3}
\end{figure}

\subsection*{Acknowledgments}  This work was initiated in LaCIM (Montreal) while Philippe Nadeau was visiting thanks to a research travel grant from the {\em Laboratoire International Franco-Qu\'eb\'ecois de Recherche en Combinatoire (LIRCO)}. The first author (CH) thanks Christophe Reutenauer for interesting conversations regarding minimality of automata and restriction to standard parabolic subgroups.  We thank an anonymous referee for suggesting that we provide more evidence for Conjectures~\ref{conj:1} and~\ref{conj:2}.

%%%%%%%%%%%%%BIBLIOGRAPHY%%%%%%%%%%%%%%%%%

\bibliographystyle{abbrv}
% use the following instead if you encounter problems 
%\bibliographystyle{alpha}
%\bibliography{masterbibliography}

\end{document}